\documentclass{amsart}
\usepackage{amsmath,amssymb,mathrsfs,amsthm}
\usepackage{array,enumerate,verbatim,hyperref,url}
\theoremstyle{plain}

\newtheorem{theorem}[subsubsection]{Theorem}

\newtheorem{proposition}[subsubsection]{Proposition}

\newtheorem{lemma}[subsubsection]{Lemma}

\newtheorem{corollary}[subsubsection]{Corollary}

\newtheorem*{theorem*}{Theorem}
\newtheorem*{thm*}{Theorem}
\newtheorem*{proposition*}{Proposition}
\newtheorem*{prop*}{Proposition}
\newtheorem*{lemma*}{Lemma}
\newtheorem*{lem*}{Lemma}
\newtheorem*{corollary*}{Corollary}
\newtheorem*{cor*}{Corollary}
\newtheorem*{sublemma*}{Sublemma}
\newtheorem*{sublem*}{Sublemma}
\newtheorem*{hypothesis*}{Hypothesis}
\newtheorem*{hyp*}{Hypothesis}
\newtheorem*{claim*}{Claim}
\newtheorem*{clm*}{Claim}

\theoremstyle{definition}
\newtheorem{definition}[subsubsection]{Definition}

\newtheorem*{definition*}{Definition}
\newtheorem*{defn*}{Definition}
\newtheorem*{setup*}{Set-up}

\theoremstyle{remark}

\newtheorem*{remark*}{Remark}
\newtheorem*{rem*}{Remark}

\usepackage{tikz}
\usetikzlibrary{calc}

\usepackage[all]{xypic}
\usepackage{xspace} 
\newcommand{\spacex}{\xspace} 
\numberwithin{equation}{section}
\usepackage{mathtools}\mathtoolsset{showonlyrefs=true} 
\usepackage{bm}

\usepackage{color}

\newcommand{\nc}{\newcommand}
\nc{\oname}{\operatorname}

\newcommand{\N}{\mathbf{N}}
\newcommand{\Q}{\mathbf{Q}}
\newcommand{\R}{\mathbf{R}}
\newcommand{\T}{\mathbf{T}}
\newcommand{\Z}{\mathbf{Z}}
\newcommand{\RZ}{\R / \Z }
\newcommand{\KR}{K\otimes _{\Q}\R}

\nc{\calO}{\mathcal O}

\newcommand{\ideala}{\mathfrak a}
\newcommand{\idealp}{\mathfrak p}
\newcommand{\idealq}{\mathfrak q}

\nc{\Nrm}{\mathbf N}
\nc{\OK}{\mathcal O _K}
\nc{\baci}{^{\times }} 
\nc{\Rpos}{\R\baci_{>0}} 

\nc{\pseudopolygrowth}[1]{e^{\sqrt{\log #1}/O(1)}} 
\nc{\abspseudopolygrowth}[1]{e^{\sqrt{\log #1}}} 
\nc{\Kpseudopolygrowth}[1]{e^{\sqrt{\log #1}/O_K(1)}} 

\nc{\pseudopolydecay}[1]{e^{-\sqrt{\log #1}/O(1)}} 
\nc{\abspseudopolydecay}[1]{e^{-\sqrt{\log #1}}} 
\nc{\Kpseudopolydecay}[1]{e^{-\sqrt{\log #1}/O_K(1)}} 

\nc{\totient}{\varphi _K} 
\nc{\vol}{\oname{Vol}}
\nc{\area}{\oname{Area}}
\nc{\ClK}{\oname{Cl}(K)} 
\nc{\Ideals}{\oname{Ideals}}
\nc{\IdealsK}{\Ideals_K}
\nc{\Primes}{\mathcal P} 

\nc{\Fou}{Fourier\spacex}
\nc{\Lip}{Lipschitz\spacex}
\nc{\Mit}{Mitsui\xspace}
\nc{\Mobius}{M{\"o}bius\spacex}
\nc{\Tera}{Ter{\"a}v{\"a}inen\spacex}

\nc{\Ideles}{\mathbf I_K} 
\nc{\idele}{id{\`e}le\spacex} 
\nc{\ideles}{id{\`e}les\spacex} 
\nc{\midele}{\mathrm{id\grave{e}le}} 

\nc{\diag}{\oname{diag}} 
\nc{\bdiag}{\oname{\mathbf{diag}}} 

\nc{\tors}{\mathrm{tors}} 
\nc{\add}{\mathrm{add}} 
\nc{\mult}{\mathrm{mult}} 

\nc{\Podual}{^\wedge} 
\nc{\podual}{^\wedge} 

\nc{\Hom}{\oname{Hom}}
\nc{\GL}{\oname{GL}}
\nc{\mtx}[1]{\begin{pmatrix}#1\end{pmatrix} }

\nc{\cube}{\square} 

\nc{\divides}{\mid} 
\nc{\notdivide}{\nmid} 
\nc{\notdiv}{\notdivide} 

\nc{\hatT}{\widehat\T} 

\nc{\directsum}{\bigoplus}
\nc{\dsum}{\directsum}
\nc{\disjointunion}{\bigsqcup} 
\nc{\dunion}{\disjointunion}

\nc{\surj}{\twoheadrightarrow} 
\nc{\inj}{\hookrightarrow} 
\nc{\surjfrom}{\twoheadleftarrow} 
\nc{\injfrom}{\hookleftarrow} 
\newcommand{\mapsfrom}{\mathrel{\reflectbox{\ensuremath{\mapsto}}}} 
\nc{\isoto}{\xrightarrow\cong}
\nc{\isofrom}{\xrightarrow[\cong]{}}
\nc{\actson}{\curvearrowright} 
\nc{\acts}{\actson}
\nc{\actedby}{\curvearrowleft} 
\nc{\acted}{\actedby}
\nc{\too}{\longrightarrow}
\nc{\xfrom}{\xleftarrow}

\DeclareMathOperator*{\residue}{res}
\nc{\ol}{\overline}
\nc{\inv}{^{-1}} 
\nc{\numberOfTerms}[2]{\underset{#2}{\underbrace{#1}}}
\newcommand{\lnorm}[1]{\left\Vert #1 \right\Vert} 

\title[Mitsui's PNT with Siegel zeros]{Notes on Mitsui's Prime Number Theorem with Siegel zeros}
\author{Wataru Kai}
\address{Mathematical Institute, Tohoku University, Aoba 6-3, Sendai 980-8578, Japan}
\email{kaiw@tohoku.ac.jp}



\newcommand{\MyAbstract}{In these notes,
we refine Mitsui's Prime Number Theorem from 1957, 
which for a number field $K$ predicts the number of prime elements in a bounded convex set in $K \otimes _{\mathbf Q} \mathbf R $ and in a specified congruence class modulo an ideal $\idealq $,
by incorporating potential Siegel zeros of Hecke $L$-functions.
This allows the norm of the modulus $\Nrm (\idealq  )$ 
to grow at a pseudopolynomial rate 
$\pseudopolygrowth{X}$
with respect to the size $X$ of the convex set
as opposed to powers of $\log X$.
The extra flexibility and precision are essential in our sequel work on linear patterns of prime elements.
We also hope that our updated exposition will make Mitsui's work accessible to a wider mathematical audience.}

\begin{document}

\begin{abstract}
    \MyAbstract
\end{abstract}

\maketitle

\setcounter{tocdepth}{1}
\tableofcontents

\section{Introduction}

\subsection{The Siegel--Walfisz prime number theorem}

The Prime Number Theorem states in one of several equivalent forms:
\begin{align}
    \sum _{0<p<X} \log (p) = X + o(X) \quad \text{ as }X\to +\infty .
\end{align}
Here the sum $\sum _{0<p<X}$ is the sum over prime numbers $p$ satisfying $0<p<X$.

A little more generally, the Siegel--Walfisz prime number theorem says for any modulus $1<q<\log ^A X$ (where $A\ge 1$ is any constant fixed in advance)
and $a\in (\Z /q\Z )\baci$, we have the estimate 
\begin{align}\label{eq:Siegel-Walfisz}
    \sum _{\substack{
        0<p<X\\ 
        p\equiv a \bmod  q
        }}
    \log (p)
    = \frac{1}{\varphi (q)} X + O_A (Xe^{-\sqrt{\log X} /O_A(1) }).
\end{align}

Formulated this way, the restriction $q<\log ^{A} X$ 
cannot be improved by today's human technology 
because the contemporary analytic number theory has not been able to rule out the existence of notorious {\it Siegel zeros} of Dirichlet $L$-functions.
Siegel zeros are potential real zeros of Dirichlet $L$-functions slightly smaller than $1$, which contradicts the Generalized Riemann Hypothesis.
This concept will be reviewed in \S \ref{sec:ideles-Hecke-L-Prime-ideal-th}.

Indeed, a Siegel zero $0<\beta <1$ modulo $q$ would give rise to a second main term $X^\beta / \beta \varphi (q) $ to \eqref{eq:Siegel-Walfisz}.
While this can be absorbed in the error term if we stay in the realm $q<\log ^{A} X$,
the second main term must be taken into account if we want to go beyond this bound.
Indeed, by incorporating the term $X^\beta / \beta \varphi (q) $ into the picture,
the modulus $q$ is allowed to be at least as large as $\pseudopolygrowth{X} $. 
Namely, for all $1\le q< \pseudopolygrowth{X} $ we have
\begin{align}
    \sum _{\substack{
        0<p<X\\ 
        p\equiv a \bmod  q
        }}
    \log (p)
    = \frac{1}{\varphi (q)} 
    \left( X 
    - \psi _{\mathrm{Siegel}}(a)\frac{X^\beta }{\beta } \right)
    + O (Xe^{-\sqrt{\log X} /O(1) }) 
\end{align}
where $\psi _{\mathrm{Siegel}}\colon (\Z /q\Z )\baci \to \{ \pm 1 \}$ is the potential Siegel character mod $q$.
See e.g.\ \cite[Theorem 5.27 on p.~122]{Iwaniec-Kowalski}.

\subsection{Mitsui's theorem and our main result}
The number field analogue of the Siegel--Walfisz theorem has already been documented.
It is known as Mitsui's prime number theorem, 
which can be formulated as follows: 

\begin{theorem}[{Mitsui \cite[Corollary and Theorem on p.~35]{Mitsui}}]
    \label{thm:Mitsui}
    Let $K$ be a number field of degree $n=[K:\Q ]$.
    Let $A>1$ be a fixed positive number.
    Let $X>1$ and consider a convex open set $C\subset \KR $ contained in the region $(\KR )_{<X}$ of size $X$ (\S \ref{sec:the_domain}).
    Let $\idealq \subset \OK $ be a non-zero ideal such that 
    \begin{align}\label{eq:Mitsui-restriction}
        \Nrm (\idealq ) < \log ^A X .
    \end{align}
    Also, choose any $\alpha \in (\OK /\idealq )\baci$.
    Then we have 
    \begin{align}\label{eq:Mitsui-statement}
        \sum _{\substack{\pi \in C, \\ \text{\upshape prime element, }\\  \pi \equiv \alpha \text{\upshape\ in }\OK /\idealq } }
        \log |\Nrm (\pi )|
        =
        \frac{w_K }{\totient (\idealq )2^{r_1}\pi ^{r_2}h_K R_K}\vol (C)
        + O_{K,A}(X^n e^{-\sqrt{\log X}/O_{K,A}(1)} ).
    \end{align}
\end{theorem}
Here, there is a canonical $\R $-linear isomorphism $\KR \cong \R ^{n}$
and $\vol (C) $ refers to the volume in the Lebesgue measure $\mu $ on $\R ^n$.

The main purpose of these notes is 
to refine Mitsui's theorem 
by incorporating the secondary main term coming from potential Siegel zeros $\beta $
of $L$-functions $L(\psi ,s)$ associated with 
Hecke characters $\psi $.
To be specific, let $\psi _{\mathrm{Siegel}}$ be the potential Siegel character mod $\idealq  $ and $\beta $ the Siegel zero (see \S \ref{sec:L-function} for these notions). 
In our main Theorem \ref{thm:main-theorem},
the right hand side of \eqref{eq:Mitsui-statement} will be refined to the expression
\begin{multline}
    \frac{w_K}{\totient (\idealq  )2^{r_1}\pi^{r_2}h_KR_K}
    \left( 
        \int _C 1-\psi _{\mathrm{Siegel}} (-;\alpha ) N(-)^{\beta -1} d\mu 
    \right) 
    + 
    O_K (X^n \Kpseudopolydecay{X} ).
\end{multline}
This modification allows us to relax the restriction \eqref{eq:Mitsui-restriction} to a pseudopolynomial growth 
\begin{align}\label{eq:we-can-allow-pseudo-poly-growth}
    \Nrm (\idealq ) < \Kpseudopolygrowth{X} .
\end{align}
%

%
%
%
%

This quantitative improvement 
is essential in our application in the sequel 
\cite{Kai23}
which extends the result of Green--Tao--Ziegler 
\cite{LinearEquations, GowersInverse} 
on linear patterns of prime numbers to general number fields. 
To illustrate what it is about, consider polynomials of degree $1$:
\begin{align}\label{eq:illustration}
    ax + by +c \in \OK [x,y],
\end{align}
where the coefficients $(a,b,c)$ vary in a given finite subset $T\subset \OK ^3$.
Under some conditions on $T$ including the hypothesis that the vectors $(a,b)\in K^2$ are pairwise linearly independent, we will be able to conclude that for any large enough convex body 
$C\subset (\KR )^2$, there are points $(x,y)\in C \cap \OK ^2$
such that the values of \eqref{eq:illustration} are prime elements of $\OK $ simultaneously for all $(a,b,c)\in T$.

This is in turn applied to a Hasse principle type problem on $K$-rational points on certain specific type of varieties as was done in \cite{Harpaz-Skorobogatov-Wittenberg} by Harpaz--Skorobogatov--Wittenberg over $K=\Q $: see \cite{Kai23} for details.

\subsection{Remarks on the formulation}
Let $\sigma _i\colon K\inj \mathbf C $ ($i=1,\dots,r_1+r_2$) be all the field embeddings modulo complex conjugation.
In Mitsui's original account \cite[Theorem on p.~35]{Mitsui}, the theorem is formulated for sets $C$ of the following form
\begin{multline}
    C= (\KR )_{<X_1,\dots ,X_{r_1+r_2}}
    :=
    \{ x\in \KR \mid 
    |\sigma _i(x)|<X_i ,\  1\le \forall i \le r_1+r_2 \} ,
\end{multline}
where $X_i$ are positive numbers
satisfying $X_i\le X_j^a$ for all $i,j$ for a prescribed constant $a\ge 1$.
Also, the error term in his formulation is
\[ O_{K,A,a}(N e^{-\sqrt{\log N } /O_{K,A,a}(1) } ), \]
where $N$ is the supremum of norms in the region $(\KR )_{<X_1,\dots ,X_{r_1+r_2}}$.
His arguments can easily be adapted to give Theorem \ref{thm:Mitsui}.
It is not clear 
if the theorem stays true for general convex bodies $C$ in $(\KR )_{<X_1,\dots ,X_{r_1+r_2}}$.

We formulated our Theorem \ref{thm:main-theorem} for the case 
$X_1 = \cdots = X_{r_1+r_2}$
because our intended application requires this case. 
It can also be shown for the case $C= (\KR )_{<X_1,\dots ,X_{r_1+r_2}}$ with essentially the same proof but we do not present the argument specifically needed for this.\footnote{Here is a hint for the interested reader. Note that $(\KR )_{<X_1,\dots ,X_{r_1+r_2}}$ is an annulus sector in the sense of Definition \ref{def:annulus-sector}.
Its logarithmic image in $\R^{r_1+r_2}$ is a cuboid in which the ratios of the side lengths are bounded by $a$. Then one runs the process of \S \ref{sec:end-of-proof}.}

In addition to prime elements, his theorem in \cite{Mitsui} contains 
statements about prime \textit{ideal numbers}---a concept 
arising from the fact that the number ring $\OK $ is not necessarily a unique factorization domain (UFD).
Our result will also be shown in this generality, replacing {\it prime ideal numbers} by {\it prime elements in ideals} (\S \ref{sec:number_rings}).
This generality is useful to draw conclusions about linear patterns of prime elements in rings of the form $\OK [ \frac{1}{f} ]$, which have real applications elsewhere such as \cite{Zywina2025a}.
%

\subsection{Outline of the method}

The method is largely the same as Mitsui's in \cite{Mitsui}.
On top of the consideration of potential Siegel zeros, there are two expositional improvements which we hope are making this text more accessible to the modern reader.
One is that we got rid of the Byzantine notion of {\it ideal numbers} (see e.g.\ 
\cite[p.~486]{Neukirch} for this notion)
and used exclusively the usual notion of {\it ideals} to formulate and prove the results.
The other is that Mitsui's highly impenetrable seemingly {\it ad hoc} computations and estimates 
have mostly been replaced by general discussions of \ideles and Fourier analysis.\footnote{Actually these two aspects have been preventing the author from being able to read through his paper to this day.}

Now we outline Mitsui's \cite{Mitsui} method. The case of quadratic fields is also sketched in references like \cite[p.~130]{Iwaniec-Kowalski} \cite[p.~318]{LangAlgNumTheory}.
Let us assume that {\it there are no Siegel zeros} 
because %
we want to focus on the main ideas here.

First, from a known zero-free region of Hecke $L$-functions, we have the following estimate of the weighted sum of prime {\it ideals} (Theorem \ref{thm:prime-ideal-theorem}). Below, $\psi $ is an arbitrary mod $\idealq  $ Hecke character\footnote{
    Let us not forget that we are assuming there is no Siegel zero. Also, in the main body of the text $\psi (\idealp)$ is written $\psi ([\idealp ])$ for good reason.}
\begin{align}\label{eq:intro-prime-ideal-theorem}
    \sum _{\substack{
        \idealp \\ 
        \text{prime to }\idealq   \\ 
        \Nrm(\idealp )<N
        } } 
        \psi (\idealp )\ \log \Nrm (\idealp )
    =
    \begin{cases}
        N  + O_K(N \Kpseudopolydecay{N}) &\text{ if }\psi \equiv 1,
        \\[3mm]
        O_K(N \Kpseudopolydecay{N}) &\text{ otherwise}.
    \end{cases} 
\end{align}

For simplicity let us further assume {\it $\OK $ is a UFD}.
If we choose an $\OK\baci $-fundamental domain $\mathcal D \subset \OK\setminus \{ 0\}  $, 
prime ideals (now all principal) correspond to unique prime elements in $\mathcal D  $.
Thus the left hand side of \eqref{eq:intro-prime-ideal-theorem} becomes 
a sum 
over prime elements $\pi \in \mathcal D  $ with $\Nrm (\pi )<N$.

\begin{figure}
    \begin{tikzpicture}[scale=0.35]

    \filldraw[lightgray!40] (0,0)
    [domain=-7:7, samples=40, variable=\y] -- plot( { sqrt((50 +((\y)^2) ) ) },\y )
    --cycle; 

    \draw[->] (-7,0)--(10,0); 
    \node at (12,0) {$\Q $};
    \draw[->] (0,-7)--(0,10); 
    \node at (0,12) {$\Q \sqrt{2} $};

    \draw[dashed, ->] (-8,-8)--(11,11);
    \draw[dashed, ->] (-8.5,8.5)--(10,-10);

    \draw (0,0)--(11,{11/(sqrt(2))});
    \draw[fill] ({sqrt(10)*sqrt(2)},{sqrt(10)*1}) circle (0.08);

    \draw (0,0)--(11,{-11/(sqrt(2))});
    \draw[fill] ({sqrt(10)*sqrt(2)},{-sqrt(10)*1}) circle (0.08);

    \node at (8.5,-2) {$\mathcal D $};

    \draw[domain=-9:10, samples=40, variable=\y] plot( { sqrt((10 +((\y)^2) ) ) },\y );
    \draw[domain=-9:9, samples=40, variable=\y] plot( { sqrt((50 +((\y)^2) ) ) },\y );
    
    \draw[domain=-7:7, samples=40, variable=\y] plot( {-sqrt((10 +((\y)^2) ) ) },\y );
    
    \draw[domain=-7:10, samples=40, variable=\x] plot( \x ,{ sqrt((10 +((\x)^2) ) ) } );
    \draw[domain=-7:9, samples=40, variable=\x] plot( \x ,{-sqrt((10 +((\x)^2) ) ) } );



\end{tikzpicture}
    \caption{\small A fundamental domain for $K=\Q (\sqrt{2})$ truncated at some norm.
    The scale is adjusted so that the point $x+y\sqrt{2}\in \Q (\sqrt{2})$ looks like $(x,y\sqrt{2})\in \R ^2$.
    The norm is $\Nrm (x+y\sqrt{2})=\left| x^2-(y\sqrt{2})^2 \right| $.
    The hyperbolas pass through points of the same norm.
    In this example we know $(\KR )/\OK\baci \cong \Rpos \times (\RZ )$.
    }\label{fig:fundamental-domain}
\end{figure}

Though not strictly needed, there is a standard, geometric way of choosing $\mathcal D  $.
For example when $K=\Q (\sqrt{2})$, an instance of $\mathcal D  $ is shown in Figure \ref{fig:fundamental-domain}.

Mod $\idealq  $ Hecke characters are 
characters 
of the mod $\idealq $ \idele class group 
$C(\idealq  )$,
which is a commutative Lie group. 
By modding out a subgroup $\cong \Rpos \cong \R $, we obtain a {\em compact} commutative Lie group $C(\idealq  )/\Rpos $.
It is a finite cover of another group
\begin{align}
    C(\idealq  )/\Rpos \surj (\KR )/(\OK\baci \cdot \Rpos ) ,
\end{align}
where the kernel is $(\OK/\idealq  )\baci $.
Fourier analysis (\S \ref{sec:Fourier}) can be applied to this torus. 
This allows us to count prime elements in a specified mod $\idealq  $ class and inside thin cones
as in Figure \ref{fig:thin-cone}.
Such cones are written as $\Rpos P$ in the main body of the text.
Thus we will obtain:\footnote{
    Let us repeat that we are assuming that there is no Siegel zero.}
\begin{align}\label{eq:intro-prime-elements-thin-cones}
    \sum _{\substack{
        \pi \text{ prime element}\\ \pi \equiv \alpha  \bmod  \idealq  \\ \pi \in \Rpos P \\ |\Nrm (\pi )|<N
    }}
    \log |\Nrm (\pi )|
    =
    \frac{\text{const.}}{\totient (\idealq  )}\ N + O_K(N \Kpseudopolydecay{N}),
\end{align}
where the constant depends only on $K$ and $\Rpos P$, and is proportional to the volume of the set 
$\Rpos P \cap \{ x\in \KR \mid |\Nrm (x)|<1 \} $.
See Corollary \ref{cor:prime-elements-between-N-and-N'}.

\begin{figure}
    \begin{tikzpicture}[scale=0.35]

\filldraw[lightgray!40] (0,0)
[domain=-7:7, samples=40, variable=\y] -- plot( { sqrt((50 +((\y)^2) ) ) },\y )
--cycle; 

    \filldraw[darkgray!50] 
    (0,0)
    [domain=3:4, samples=40, variable=\y] -- plot( { sqrt((50 +((\y)^2) ) ) },\y )
    --cycle; 

    \draw[->] (-7,0)--(10,0); 
    \node at (12,0) {$\Q $};
    \draw[->] (0,-7)--(0,10); 
    \node at (0,12) {$\Q \sqrt{2} $};

    \draw[dashed, ->] (-8,-8)--(11,11);
    \draw[dashed, ->] (-8.5,8.5)--(10,-10);

    \draw (0,0)--(11,{11/(sqrt(2))});
    \draw[fill] ({sqrt(10)*sqrt(2)},{sqrt(10)*1}) circle (0.08);

    \draw (0,0)--(11,{-11/(sqrt(2))});
    \draw[fill] ({sqrt(10)*sqrt(2)},{-sqrt(10)*1}) circle (0.08);

        \node at (8.5,-2) {$\mathcal D $};

    \draw (0,0)--(10,5);
    \draw (0,0)--(10,3.9);
    \node at (12,4.5) {$\Rpos P$};

    \draw[domain=-9:10, samples=40, variable=\y] plot( { sqrt((10 +((\y)^2) ) ) },\y );
    \draw[domain=-9:9, samples=40, variable=\y] plot( { sqrt((50 +((\y)^2) ) ) },\y );
    
    \draw[domain=-7:7, samples=40, variable=\y] plot( {-sqrt((10 +((\y)^2) ) ) },\y );
    
    \draw[domain=-7:10, samples=40, variable=\x] plot( \x ,{ sqrt((10 +((\x)^2) ) ) } );
    \draw[domain=-7:9, samples=40, variable=\x] plot( \x ,{-sqrt((10 +((\x)^2) ) ) } );



\end{tikzpicture}
    \caption{\small Fourier analysis allows us to focus on prime elements inside thin cones like this (colored dark gray).}\label{fig:thin-cone}
\end{figure}

In \S \ref{sec:end-of-proof}, we will use \eqref{eq:intro-prime-elements-thin-cones} to count prime elements in a given convex set $C$ by approximating it with segments of thin cones (Figure \ref{fig:convex}).
This approximation can be done within the claimed error.

Management of the error, notably its independence from $\idealq  $,
requires 
somewhat routine but careful analysis, 
which is done in Appendices \ref{sec:Dirichlet}--\ref{sec:Euclid}.


\begin{figure}
    \begin{tikzpicture}[scale=0.45]

    \filldraw[darkgray!50] 
    (4.72,2.36)
    [domain=2.36:1.9, samples=40, variable=\y] -- plot( { sqrt((16.71 +((\y)^2) ) ) },\y )
    [domain=4.4:5.48, samples=40, variable=\y] -- plot( { sqrt((90 +((\y)^2) ) ) },\y )
    --cycle; 

    \filldraw[darkgray!80] 
    (5.05,{sqrt(5.05^2-21)})
    [domain=2.36:1.62, samples=40, variable=\y] -- plot( { sqrt((21 +((\y)^2) ) ) },\y )
    [domain=3.4:4.4, samples=40, variable=\y] -- plot( { sqrt((92 +((\y)^2) ) ) },\y )
    --cycle; 

    \draw[domain=-3:9, samples=40, variable=\y] plot( { sqrt((90 +((\y)^2) ) ) },\y );
    \draw[domain=-3:9, samples=40, variable=\y] plot( { sqrt((16.71 +((\y)^2) ) ) },\y );

    \draw[domain=2:5, samples=40, variable=\y] plot( { sqrt((92 +((\y)^2) ) ) },\y );
    \draw[domain=1:4, samples=40, variable=\y] plot( { sqrt((21 +((\y)^2) ) ) },\y );

    \node at (7,10) {$C$};
    \draw (7,5) circle (3.5);

    \draw[->] (-3,0)--(11.5,0); 
    \node at (12,0) {$\Q $};
    \draw[->] (0,-3)--(0,10); 
    \node at (0,12) {$\Q \sqrt{2} $};

    \draw[dashed, ->] (-3,-3)--(11,11);
    \draw[dashed] (-3,3)--(3,-3);

    \draw (0,0)--(13,7.8);
    \draw (0,0)--(12,6);
    \draw (0,0)--(12,5.04);
    \draw (0,0)--(12,4);
    \draw (0,0)--(12,3);



\end{tikzpicture}
    \caption{\small Approximation of a convex body $C$ by segments of thin cones}\label{fig:convex}
\end{figure}

\subsection{General notation and terminology}\label{sec:notation}

\subsubsection{Multiplicative groups of complex numbers}\label{sec:multiplicative-groups}

Here are some subgroups of $\mathbf{C}\baci $ that we frequently use in this paper.

\begin{itemize}
\item 
We write $\Rpos $ for the multiplicative group of positive real numbers. We know $\Rpos \cong \R $ via the logarithm.

\item 
We write $S^1\subset \mathbf{C} \baci$ for the subgroup of complex numbers of absolute value $1$.
It is canonically isomorphic to $\R /\Z $ by the map 
$(\theta  \bmod{\Z })\mapsto e^{2\pi i \theta }$.

\item 
We know 
\begin{align}
&\R\baci\cong \Rpos \times \{ \pm 1\}\cong \R \times \{\pm 1\}  \text{ and }
\\ 
&\mathbf{C}\baci\cong \Rpos \times S^1\cong \R \times (\R /\Z ),
\end{align}
where we choose the projection $\mathbf{C}\baci \to \Rpos $ to be $z\mapsto |z|^2$.
\end{itemize}

\subsubsection{Number rings}\label{sec:number_rings}
We write $\IdealsK $ for the multiplicative monoid of non-zero ideals of $\OK $.
We write $h_K:= \# \oname{Cl} (K)$ for the class number.
For each class $\lambda \in \oname{Cl} (K)$, we fix a representing fractional ideal $\ideala _\lambda \subset K$.
For the trivial class $0\in \oname{Cl}(K)$, let us choose $\ideala _{0}:= \OK $.
Let $\ideala_\lambda\inv =\{ \alpha \in K\mid \alpha \ideala_\lambda \subset \OK \}$ be the inverse fractional ideal.
Suppose an $\OK\baci$-fundamental domain $\mathcal D  $ for $(\KR )\baci $ is given.
We have bijections 
\begin{align*} 
    \Bigl\{ (0)\neq \ideala\subset \OK \text{ ideals }\Bigm| %
    [\ideala ]= \lambda \text{ in }\ClK  \Bigr\}
    &\ \cong \ 
    (\ideala_\lambda\inv \setminus \{ 0\}  ) / \OK \baci \ \cong\ \ideala _\lambda\inv \cap \mathcal D  
    \\ 
    \ideala = a\ideala _\lambda 
    &\ \mapsfrom\ a .
\end{align*}

It follows that there is a bijection 
\begin{align}\label{eq:parametrization}
    \IdealsK \cong \bigsqcup _{\lambda \in \ClK } \ideala _\lambda\inv \cap \mathcal D  .
\end{align}

Let us call an element $\pi \in \ideala $ of a non-zero fractional ideal a {\it prime element} of $\ideala $ if $\pi \ideala \inv $ is a non-zero prime ideal (where $\ideala\inv $ is the inverse fractional ideal of $\ideala $).\footnote{
    An algebro-geometric justification for this piece of terminology is that for any integral scheme $X$ and a line bundle $\mathcal L $ on it, we could call a non-zero section $s\colon \calO _X \to \mathcal L $ {\it prime} if its zero locus $(s)\subset X$ is an integral scheme.} %
Denote by $\Primes (\ideala )$ the set of prime elements of $\ideala $.
In this paper, sums and products $\bigoplus _{\idealp }, \sum _{\idealp }, \prod _{\idealp}$ indexed by the letter $\idealp $ will always mean those indexed by non-zero prime ideals, often subject to some additional conditions.
Likewise, sums and products $\bigoplus _{\pi }, \sum _{\pi }, \prod _{\pi}$ indexed by $\pi $ will mean those indexed by prime elements.

Euler's totient function $\totient $ for $K$ is defined as the function 
$\ideala \mapsto \# (\OK / \ideala )\baci $,
$\IdealsK \to \N $.

\subsubsection{Landau's $O(-)$ symbol}\label{sec:Landau's-symbol}
The symbol $O_{a,b,c}(1)$ denotes a complex value whose absolute value is bounded from above by a constant $C_{a,b,c}>0$ depending only on the parameters $a,b,c$.
The specific value can be different from occasion to occasion.
In particular, an $O(1)$ means a complex value whose absolute value is bounded by an absolute constant independent of any parameter even if some parameters are at play.

When $f(t)$ is a positively valued function of $t$,
the expression $O_{a,b,c}(f(t))$ means $O_{a,b,c}(1)\cdot f(t)$. Often $t$ will be a positive real parameter, and the bound may be valid only when $t$ is large enough.
We allow this threshhold to depend on the parameters $a,b,c$.
We shall indicate it if the threshhold depends on more (or less) parameters.

We make an exception of allowing ourselves to use $O_K(-)$ even when the bounds depend on additional choices related to $K$,
as long as those choices are independent of the problem at hand and can be made once the number field $K$ is given.
Such choices include 
an isomorphism $\OK \baci \cong (\Z / w\Z) \times \Z ^{r}$ and representatives $\ideala _\lambda $ ($\lambda \in \ClK $) for the ideal class group.

The expression $f(t)\ll _{a,b,c} g(t)$ means 
$f(t)=O_{a,b,c}(g(t))$.
The expression $f(t)\asymp _{a,b,c} g(t)$ means that $f(t)\ll _{a,b,c} g(t)$ and $g(t)\ll _{a,b,c}f(t)$ both hold.

It goes without saying that a statement of the form 
\begin{itemize}
    \item If $A=B/O_{a,b,c}(1)$, then $X=O_{d,e,f}(Y)$, or 
    \item 
    If $A\ll _{a,b,c}B$, then $X\ll _{d,e,f} Y$
\end{itemize}
should be read as the statement that there are large enough positive constants $C_{a,b,c}>1$ and $C_{d,e,f}>1$
depending only on the parameters indicated in subscript such that the following holds:
\begin{quotation}
    If $|A|<B/ C_{a,b,c}$, then $|X|<C_{d,e,f}Y$.
\end{quotation}
Observe that the larger the constants $C_{a,b,c},C_{d,e,f}$ are, the more restrictive the hypothesis is and the weaker the conclusion. So the statement gets more likely to hold.

\subsubsection{Usage of some letters}

The letter $X$ is often used to indicate the length which characterizes the region in $\KR $ in which we search for prime elements.

The letter $N$ is used to bound the norm of elements. It is often set to be $N:=X^n$,
where $n=[K:\Q ]$.

The letters $M,Y>1$ will be used to denote generic positive parameters which grow at a pseudopolynomial rate $e^{\sqrt{\log N} / O_K (1)} $.
The specific value of $O_K(1)$ may vary from occasion to occasion.

\subsection*{Acknowledgments}
I thank Federico Binda, Hiroyasu Miyazaki and Rin Sugiyama for fruitful on-line and in-person discussions as well as feedback to earlier versions of this paper. 
I was supported by Japan Society for the Promotion of Science (JSPS) 
through JSPS Grant-in-Aid for Young Scientists (numbers JP18K13382, JP22K13886 and JP26K06736).
Much of this work was done while I was staying at the University of Milan
as a JSPS Overseas Research Fellow in 2022--24. 
I thank my colleagues there for the pleasant working environment.

\section{Recollection of the Prime Ideal Theorem}\label{sec:ideles-Hecke-L-Prime-ideal-th}

Let us recall the notion of Hecke characters, the associated $L$-functions,
their known zero-free regions
and the Prime Ideal Theorem \ref{thm:prime-ideal-theorem} that results.

\subsection{Hecke characters}\label{sec:Hecke}
We denote the group of \ideles by $\Ideles $.
A {\it Hecke character} is a continuous homomorphism 
\begin{align}
    \psi \colon \Ideles /K\baci \to S^1 .
\end{align}
In the literature it is also called a {\it Gr{\"o}{\ss}encharakter}.
Every such $\psi $ factors through a quotient of the following form for some non-zero ideal $\idealq = \prod _{\idealp \text{ finitely many}} \idealp ^{n_{\idealp }}$:
\begin{align}
    \psi \colon \Ideles /K\baci 
    &\surj 
    \left( (\KR )\baci \times 
    \bigoplus _{\idealp | \idealq }
        K\baci / (1+\idealp ^{n_{\idealp}} \calO _{K,\idealp } )
    \times 
    \bigoplus _{\idealp \notdiv \idealq }
        \Z 
    \right) /K\baci
    \\ 
    &\too S^1 .
\end{align}
In this case we say $\psi $ is a {\it mod $\idealq $ Hecke character}.
It is called a {\it primitive} mod $\idealq $ Hecke character if it is not a mod $\idealq '$ Hecke character for any other divisor $\idealq ' | \idealq $, $\idealq ' \neq \idealq $. 
In this case $\idealq $ is called the {\it conductor} of $\psi $.
See \cite[XV-\S 3]{Cassels-Froehlich},
\cite[VII-\S 3 and XVI-\S 7]{LangAlgNumTheory} 
or 
\cite[VI-\S 1 and VII-\S 6]{Neukirch} for more background.

Let us write this quotient (the so-called \idele class group mod $\idealq $) in the middle as 
\begin{align}
    C(\idealq ) = \widetilde C(\idealq ) /K\baci 
    := 
    \left( (\KR )\baci \times 
    \bigoplus _{\idealp | \idealq }
        K\baci / (1+\idealp ^{n_{\idealp}} \calO _{K,\idealp } )
    \times 
    \bigoplus _{\idealp \notdiv \idealq }
        \Z 
    \middle) 
    \right/ K\baci
    .
\end{align}

\subsection{The $L$-function}\label{sec:L-function}
Let $\psi $ be a Hecke character with conductor $\idealq $.
For prime ideals $\idealp \notdivide \idealq $, let $[\idealp ]\in \widetilde C(\idealq ) $ be the element 
whose $\idealp $-component is $1$ and the other components are trivial.
Sometimes the symbol $[\idealp ]$ is used to mean its image to $ \bigoplus _{\idealp' \notdivide \idealq  }\Z $ as well.
Multiplicatively we can define $[\ideala ]\in \widetilde C(\idealq )$ or $\bigoplus _{\idealp \notdivide \idealq  }\Z $ for every ideal $\ideala $ prime to $\idealq $.

The {\it Hecke $L$-function} associated to $\psi $ is defined by (recall $\idealq  $ is the conductor)
\begin{align}
    L(s,\psi )
    &:= \sum _{\ideala \text{ ideal prime to }\idealq }
    \frac{\psi ([\ideala ])}{\Nrm (\ideala )^s}
    \\ 
    &= \prod _{\idealp \notdivide \idealq }
    \frac{1}{1-\frac{\psi ([\idealp ])}{\Nrm (\idealp )^s}}
\end{align}
for $\Re (s)>1$.
$L(s,\psi )$ is known to be a meromorphic function on $\mathbf{C} $ (see e.g.\ \cite[Main Theorem 4.4.1 and p.~346]{Cassels-Froehlich} \cite[Theorem 12 on p.~295 and p.~299]{LangAlgNumTheory} or \cite[VII-\S 8 (8.6)]{Neukirch}).

\subsection{The zero-free region and prime ideal theorem}

A classical zero-free region is known:
\begin{theorem}[{zero-free region}]
    \label{thm:zero-free-region}
    There is a positive constant $0<c_K<1$ %
    such that the following holds for every non-zero ideal $\idealq  \subset \OK $. 
    \begin{enumerate}
        \item 
        For all but possibly one mod $\idealq $ Hecke character $\psi $ (not necessarily primitive), the $L$-function $L(s,\psi )$ does not have zeros or poles in the region 
        (where $s=\sigma +it$ as usual)
        \begin{align}
            \sigma > 1- \frac{c_K}{\log ( \Nrm(\idealq ) (|t|+4) )} .
        \end{align} 
        \item 
        The potential exceptional $\psi $, if it exists, has to be {\it real}
        (i.e., takes values in $\{ -1,+1 \}$). 
        This $\psi $ is called the {\em Siegel} (or {\em exceptional}) character mod $\idealq $ and denoted by $\psi _{\mathrm{Siegel}}$. 
        \item 
        The $L$-function $L(\psi _{\mathrm{Siegel}},s)$ can have at most one zero $\beta $ and the zero has to be real:
        \begin{align}
            1- \frac{c_K}{\log ( 4\Nrm(\idealq )  )} < \beta <1 .
        \end{align}
        The zero $\beta $ is called the {\em Siegel} (or {\em exceptional}) zero mod $\idealq $.
    \end{enumerate}
\end{theorem}
\begin{proof}
    See e.g.\ \cite[Theorem 5.35]{Iwaniec-Kowalski} or \cite[Lemma 2.3]{Lagarias-Montgomery-Odlyzko}, and \cite[Theorem 11.7]{Montgomery-Vaughan} or \cite[Theorem 1.9]{Weiss83}.
\end{proof}

It is known that an $L$-function with a zero-free region as in Theorem \ref{thm:zero-free-region} is accompanied by a Prime Ideal Theorem where the prime ideals are counted with weights.
To us, the pseudopolynomial control of the error term in the following theorem is important. 

\begin{theorem}[{Prime Ideal Theorem}]
    \label{thm:prime-ideal-theorem}
    Let $N>1$ be a positive number and $\idealq \in \IdealsK $ be an ideal with norm 
    \begin{align}
        \Nrm (\idealq ) < \Kpseudopolygrowth{N} .
    \end{align} 
    Let $\psi $ be a mod $\idealq $ Hecke character.
    Then we have 
    \begin{align}
        \sum _{\substack{
            \idealp \\ 
            \idealp \notdivide \idealq ,\ \Nrm (\idealp )<N}
        }
        \psi ([\idealp ]) \log \Nrm (\idealp ) 
        = N\cdot 1_{\psi \equiv 1}
        - \frac{N^\beta }{\beta }1_{\psi \mathrm{:Siegel}}
        + O_K(N\Kpseudopolydecay{N}),
    \end{align}
    where $\beta $ is the potential Siegel zero mod $\idealq $.
    
    The terms $1_{[\cdots ]}$ are present only when the condition in the subscript is true
    and equal the constant function $1$ in such cases.
\end{theorem}
\begin{proof}
    See e.g.\ \cite[Theorem 5.13]{Iwaniec-Kowalski}.
\end{proof}

\section{The statement and first steps of the proof}

Before stating our main result, we have to introduce a few pieces of notation. 
Let $\ideala ,\idealq  \subset \OK $ be non-zero fractional ideals. 
Write $(\ideala /\idealq \ideala )\baci \subset \ideala /\idealq \ideala $
for the set of residue classes which can be a single generator of $\ideala /\idealq \ideala $ as an $\OK $-module.
By Proposition \ref{prop:chart}, we have a canonical set map 
$(\KR )\baci \times (\ideala /\idealq  \ideala )\baci \to C(\idealq  )$.
Thus for any function $\psi $ on $C(\idealq  )$, we may consider its pullback to $(\KR )\baci \times (\ideala /\idealq  \ideala )\baci$.
We write this restriction as $(x,\alpha )\mapsto \psi (x;\alpha )$.
In particular if we fix a residue class $\alpha \in (\ideala /\idealq  \ideala )\baci $,
we obtain a function 
\begin{align}
    \psi (-;\alpha )\colon (\KR)\baci  \to \mathbf{C} .
\end{align}

\subsection{The statement}

\begin{theorem}\label{thm:main-theorem}
    Let $K$ be a number field of degree $n=[K:\Q ]= \dim _\Q K$.
    Fix a non-zero fractional ideal $\ideala \subset \OK $.

    Let $X >1 $ be a real number.
    Set $N:=X^n$. 
    
    Let $C\subset (\KR )_{<X\Nrm (\ideala )^{1/n}}$ be a convex open set and $\idealq $ be an ideal whose norm obeys 
    a bound
    \begin{align}\label{eq:norm-bound-idq}
        \Nrm (\idealq ) < \Kpseudopolygrowth{N}
    \end{align}
    with a large enough constant $O_K(1)>1$. 
    Let $\alpha \in (\ideala /\idealq \ideala)\baci $.
    Then we have 
    \begin{multline}\label{eq:main-theorem}
        \sum _{
            \substack{\pi \in C\cap \ideala ,\\ \pi = \alpha \text{ in }\ideala / \idealq \ideala }}
        \log \Nrm (\pi \ideala \inv )
        =
        \\[10pt] 
        \frac{1}{\Nrm (\ideala )\totient (\idealq )}\frac{w_K }{2^{r_1}\pi^{r_2} h_KR_K}
        \int _C 1- \psi _{\mathrm{Siegel}}(-;\alpha ) \Nrm ((-)\ideala\inv )^{\beta -1} 
        d\mu _{\add}
        \\[10pt] 
        + O_{K}\left( N 
        \Kpseudopolydecay{N } 
        \right) ,
    \end{multline}
    where by convention we ignore the term $\psi _{\mathrm{Siegel}}(-;\alpha ) \Nrm ((-)\ideala\inv )^{\beta -1}$ if there is no Siegel zero mod $\idealq $;
    in this case the integral is $\vol _{\add }(C)$.
\end{theorem}

Our convention for the measure $\mu _{\add },\vol _{\add }$ is in \S \ref{sec:measures}.

Though this is not needed in the sequel,
the coefficient 
${w_K }/(2^{r_1}\pi^{r_2} h_KR_K)$
is equal to 
${2^{r_2}}/(\residue_{s=1}(\zeta _K(s)) \sqrt{|D_K|} )$
by Class Number Formula \cite[Theorem 5 on p.~161]{LangAlgNumTheory},
where $D_K$ is the discriminant of $K$.

Requirements for the $O_K(1)$ constant in \eqref{eq:norm-bound-idq} will be specified in \S\S \ref{sec:choosing-Y-and-M}, \ref{sec:choosing-Y2-and-bound-q}.

\subsection{First reduction}

Here we show that in Theorem \ref{thm:main-theorem} we may assume $\ideala $
equals the pre-chosen representative $\ideala_\lambda $ of its ideal class $\lambda $ without loss of generality.
This will also allow us to simplify the range of the convex set $C$ to $C\subset (\KR )_{<X }$, omitting $\Nrm (\ideala _\lambda )^{1/n}$, because 
$\Nrm (\ideala _\lambda )$ is an $O_K(1)$ when representatives $\ideala _\lambda $ are fixed beforehand.
The reader who is not interested in Theorem \ref{thm:main-theorem} for a general fractional ideal $\ideala $ may of course skip 
over this step
to \S \ref{sec:prime-ideals-to-elements}. 

By the bijection \eqref{eq:parametrization} 
\begin{align}
    \IdealsK \cong \bigsqcup _{\lambda \in \ClK} (\ideala_\lambda\inv \setminus \{ 0\}  )/\OK\baci 
\end{align}
there are an index $\lambda $ and an element $a_0 \in \ideala_\lambda \inv $ such that $\ideala = a_0\ideala _\lambda $.
We have 
\begin{align}\label{eq:compatibility-ideala-and-lambda-0}
    \Nrm (\ideala )= \Nrm (a_0)\Nrm (\ideala_\lambda ).
\end{align}

We argue that Theorem \ref{thm:main-theorem} for $\ideala $ follows from the case of $\ideala _\lambda $ if we multiply some parameters by $a_0\inv \in K\baci \subset (\KR )\baci$.

\subsubsection{The left hand side}
Via the correspondence of prime elements 
$\Primes (\ideala )\isoto \Primes (\ideala_\lambda )$, $\pi \mapsto a_0\inv \pi =:\pi '$, the left hand sides of \eqref{eq:main-theorem} for $\ideala $ and $\ideala_\lambda $ are equal:
\begin{align}\label{eq:compatibility-ideala-and-lambda-1}
    \sum _{\substack{\pi \in C\cap \ideala, \\ \pi \equiv \alpha \text{ in }\ideala /\idealq \ideala }}
    \log N(\pi \ideala \inv )
    =
    \sum _{\substack{\pi ' \in a_0\inv C\cap \ideala_\lambda ,\\ \pi' \equiv a_0\inv \alpha \text{ in }\ideala_\lambda  /\idealq \ideala_\lambda }}
    \log N(\pi '\ideala _\lambda \inv ) .
\end{align}
There is a constant $A=O_K(1)$ independent\footnote{
    The independence of $A$ from $\ideala $ is not needed if one is fine with the 
    dependence on $\ideala$ of the
    constants in the error term in the theorem.
    But this independence plays a role in \cite{Kai23}.
    } 
of $\ideala $ such that 
$a_0 $ can be taken from $(\KR)_{<A\Nrm (\ideala \ideala_\lambda\inv )^{1/n }}$,
see e.g.\  \cite[Lemma 4.11]{KMMSY} or \cite[Lemma 4.2]{Maynard}, the latter of which is stated for principal ideals but whose proof is valid for the general case if appropriately read.
It follows that 
$a_0\inv C\subset (\KR )_{< A' X\Nrm (\ideala _\lambda )^{1/n}}$ for another $A'={O_K(1)} $.
Write $X':= A' X$
and 
\begin{align}\label{eq:compatibility-ideala-and-lambda-2}
    N' := (X')^{ n} \asymp _{K} N
    .
\end{align}

\subsubsection{The right hand side}
On the right hand sides, since the Jacobian of the multiplication map 
$[a_0]\colon a_0\inv C \to C$ is $\Nrm (a_0)$ we have the compatibility of measures 
$[a_0]^* d\mu_{\add} = N(a_0)d\mu _{\add}$.

We can compute the pullback of the norm function as $[a_0]^* N((-)\ideala \inv )= N((-)\ideala _\lambda \inv )$.

We have $a_0\inv \alpha \in (\ideala_\lambda /\idealq  \ideala_\lambda )\baci $. 
By the commutative diagram 
\begin{align}
    \xymatrix@R=10pt{
        (\KR )\baci \times (\ideala/\idealq  \ideala)\baci  \ar[dr]
        \\ 
        & C(\idealq  )
        \\ 
        (\KR )\baci \times (\ideala_\lambda /\idealq  \ideala_\lambda) \baci \ar[ur]
        \ar[uu]_{[a_0]}^{\cong }
    }
\end{align}
we have 
$[a_0]^*\psi (-;\alpha )=\psi (-;a_0\inv \alpha )$.

It follows that 
\begin{align}\label{eq:compatibility-ideala-and-lambda-3}
    &\int _C 1
    -\psi _{\mathrm{Siegel}} (-;\alpha ) 
    N((-)\ideala\inv )^{\beta -1} d\mu_{\add} \\ 
    =
    \Nrm (a_0)&\int _{a_0\inv C} 1
    -\psi _{\mathrm{Siegel}} (-;a_0\inv \alpha ) 
    N((-)\ideala_\lambda \inv )^{\beta -1} d\mu_{\add} .
\end{align}

\subsubsection{End of reduction to $\ideala_\lambda $}
Now if we assume the validity of Theorem \ref{thm:main-theorem} for $\ideala_\lambda $ and the convex body $a_0\inv C 
\subset (\KR)_{< X' \Nrm (\ideala _\lambda )^{1/n}}$,
we have
\begin{multline}
    \sum _{\substack{\pi ' \in a_0\inv C\cap \ideala_\lambda ,\\ \pi' \equiv a_0\inv \alpha \text{ in }\ideala_\lambda  /\idealq \ideala_\lambda }}
    \log N(\pi '\ideala _\lambda \inv )
    =\\ 
    \frac{w_K }{\Nrm (\ideala_\lambda )\totient (\idealq )2^{r_1}\pi^{r_2}h_KR_K }
        \int _{a_0\inv C} 1- \psi _{\mathrm{Siegel}}(-;a_0\inv \alpha ) \Nrm ((-)\ideala_\lambda \inv )^{\beta -1} 
        d\mu_{\add} 
        \\[10pt] 
        + O_{K}
        (
            N'
            \Kpseudopolydecay{N'}
        )
        .
\end{multline}
By 
\eqref{eq:compatibility-ideala-and-lambda-0}
\eqref{eq:compatibility-ideala-and-lambda-1}
\eqref{eq:compatibility-ideala-and-lambda-2}
\eqref{eq:compatibility-ideala-and-lambda-3},
this is equivalent to Theorem \ref{thm:main-theorem} for $\ideala $ and $C\subset (\KR)_{<X\Nrm (\ideala )^{1/n}}$.

This reduces the problem to the case of $\ideala = \ideala_\lambda $ for some $\lambda \in \ClK $.

\subsection{From prime ideals to prime elements}\label{sec:prime-ideals-to-elements}

We already know how many prime {\it ideals} there are within bounded norms (Theorem \ref{thm:prime-ideal-theorem}).
Here is the first step to deduce information on the number of prime {\it elements} in bounded regions.

Let the situation be as in Theorem \ref{thm:main-theorem} and let $\psi $ be a mod $\idealq $ Hecke character.

\subsubsection{Prime ideals and prime elements}
Let $\mathcal D  \subset (\KR )\baci$ be an arbitrary $\OK \baci$-fundamental domain.

We temporarily use 
the bijection \eqref{eq:parametrization}
in the form 
\begin{align}
    \IdealsK \cong \dunion _{\lambda \in \ClK } \ideala _\lambda \cap \mathcal D  .    
\end{align}
Namely we use the complete set of representatives $\{ \ideala_\lambda\inv \} _\lambda $ as opposed to $\{ \ideala _\lambda \} _\lambda $.
It follows that for every $\idealp $, there are a unique $\lambda \in \ClK $ and a $\pi \in \ideala_\lambda \cap \mathcal D  $
such that $\idealp = \pi \ideala_\lambda\inv  $.
In this situation $\pi $ is a prime element of $\ideala_\lambda $ by the very definition (\S \ref{sec:number_rings}) of this notion.
This establishes the bijection 
\begin{align}
    \{ \idealp \subset \OK \mid \text{ non-zero prime ideals } \}
    \cong \dunion _{\lambda \in \ClK } \Primes (\ideala_\lambda )\cap \mathcal D  .
\end{align}

\subsubsection{To subtract a diagonal image}
Now let $\idealp $ be a non-zero prime ideal {\it not} dividing $\idealq $ and $\pi \in \ideala_\lambda \cap \mathcal D  $ the corresponding prime element.
Let us write $\bdiag (\pi )\in \widetilde C (\idealq )$ for the canonical diagonal image
\begin{align}
    \bdiag (\pi )  \in 
    \widetilde C (\idealq )=
    (\KR )\baci \times 
    \bigoplus _{\idealp _1| \idealq }
    K\baci / (1+\idealp_1 ^{n_{\idealp_1}} \calO _{K,\idealp_1 } ) 
    \times \dsum _{\idealp _2 \notdivide \idealq } \Z .
\end{align}
We know $v_{\idealp_2}(\pi )=0$ for $\idealp_2 \notdivide \idealq  $ different from $\idealp $.  
This implies that $\bdiag (\pi )= (\pi ; \diag (\pi ); [\idealp ] )$,
where $\diag (\pi )\in \bigoplus _{\idealp_1 \divides \idealq  } K\baci /(1+\idealp_1^{n_{\idealp_1}}\calO_{K,\idealp_1 })$ is also the diagonal image, which we will simply write as $\pi $ in the sequel.
Since the Hecke character $\psi $ kills the image of $K\baci$,
we know 
\begin{align}
    \psi ([\idealp ]) &= \psi ([\idealp ] \bdiag(\pi )\inv )
    \\ 
    &= \psi ( \pi \inv ; \pi \inv ; 0 )
    \\ 
    &= \ol{\psi }( \pi ; \pi ; 0 ).
\end{align}

By changing the notation $\psi \leftrightarrow \ol\psi $, from Theorem \ref{thm:prime-ideal-theorem}
we deduce:

\begin{corollary}\label{cor:prime-ideal-theorem}
    Under the notation of Theorem \ref{thm:prime-ideal-theorem} and this \S \ref{sec:prime-ideals-to-elements}, we have 
    \begin{align}\label{eq:cor:prime-ideal-theorem}
        &\sum _{\lambda\in \ClK }
        \sum _{\substack{\pi \in \ideala_\lambda \cap \mathcal D  ,\\  
                \Nrm (\pi \ideala_\lambda\inv )< N}} 
        \psi (\pi ;\pi ;0) 
        \log \Nrm (\pi\ideala_\lambda\inv )
        \\
        =
        &N\cdot 1_{\psi \equiv 1}
        - \frac{N^\beta }{\beta }1_{\psi \mathrm{:Siegel}}
        +O_K(N\Kpseudopolydecay{N}) .
    \end{align} 
\end{corollary}

\section{Main theorem for thin cones---Fourier analysis}

Here we will combine Fourier analysis recalled in \S\S \ref{sec:Fourier}--\ref{sec:idele-class-group-as-a-torus} and Corollary \ref{cor:prime-ideal-theorem}
to prove Theorem \ref{thm:main-theorem} for a special type of subsets (Corollary \ref{cor:main-thm-for-Rpos.P}).

Let the situation be as in Theorem \ref{thm:main-theorem}.

By Proposition \ref{prop:chart} we have a chart 
\begin{align}
    C(\idealq  )/ \R \baci_{>0} 
    &\cong \dunion _{\lambda\in\ClK} 
    \frac{ 
            ((\KR )\baci /\Rpos ) 
            \times 
            (\ideala _\lambda /\idealq  \ideala _\lambda )\baci 
            }
            {
                \OK\baci 
                }
                .\end{align}
                Let $\lambda \in \ClK$ be an ideal class
                and let $\alpha \in (\ideala_\lambda /\idealq  \ideala_\lambda )\baci $.

\subsection{Thin cones}
Let $P\subset (\KR)\baci /\Rpos $ be a connected open 
set such that the map 
$P\to (\KR )\baci /(\OK\baci \Rpos )$
is injective.
There is an 
$\OK\baci $-fundamental domain 
$\mathcal D  \subset (\KR )\baci $ containing $\Rpos P$. Fix any such $\mathcal D  $.

It follows that 
$P\times \{\alpha \} \to C(\idealq  )/\Rpos $ 
is also injective.
Write $P\times \{\alpha \} \subset C(\idealq  )/\Rpos $ for the image as well.
Write $(\Rpos P) \times \{\alpha \}$ for the corresponding subset of either 
$(\KR )\baci \times (\ideala_\lambda /\idealq  \ideala_\lambda )\baci $ or $C(\idealq  )$.

\subsubsection{Applying Fourier analysis}\label{sec:if-P-small-enough}
Suppose further $P$ is convex and small enough, depending on $\totient (\idealq  )$, to satisfy the hypothesis of Proposition \ref{prop:Fourier-on-idele-class-group}.
The 
term {\it thin cones} refers to subsets of the form $\Rpos P$ for such a $P$.
We then have a decomposition of the indicator function $1_{P\times \{ \alpha \}}$ on $C(\idealq  )/\Rpos $ as in Proposition \ref{prop:Fourier-on-idele-class-group}:
\begin{align}\label{eq:Fourier-on-idele-applied}
    1_{P\times \{ \alpha \}}
    =
    \sum _{\psi \in \Psi \subset (C(\idealq  )/\R \baci_{>0})}
    c_\psi \psi 
    + G+H .
\end{align}
Now we take the sum $\sum _{\psi \in \Psi }c_\psi $ of \eqref{eq:cor:prime-ideal-theorem}.
Combined with \eqref{eq:Fourier-on-idele-applied} we obtain the following
\begin{multline}\label{eq:Fourier-applied-to-prime-ideals}
    \sum _{\substack{
        \pi\in \ideala _\lambda \cap \mathcal D  ,\\ \Nrm (\pi \ideala_\lambda\inv )< N, \\ (\pi ;\pi ;0) \in (\Rpos P)\times \{ \alpha \} } }
    \log \Nrm (\pi \ideala_\lambda\inv )
    +G((\pi ;\pi ;0 )) +H((\pi ;\pi ; 0))
    \\ 
    = 
    \frac{\vol _{\midele }(P\times \{\alpha \})}{\vol_{\midele } (C(\idealq  )/\R \baci_{>0})}
    \Bigl( N (1+O_n(1/M))
        \\
        - 
        \frac{N^\beta }{\beta }
        (\psi _{\mathrm{Siegel}}(P\times \{\alpha \}) + O_n(1/M)) 
    \Bigr)
    \\
    +O_K(\totient (\idealq ) Y^{n-1})O_K(N\Kpseudopolydecay{N}) .
\end{multline}
Here, note that as $\psi _{\mathrm{Siegel}}$ takes values in the discrete set $\{\pm 1\}$ and $P\times \{\alpha \}$ is connected, %
the value $\psi_{\mathrm{Siegel}}(P\times \{\alpha\})\in \{\pm 1\}$ is well defined.

The parameters $Y,M>1$ are from Proposition \ref{prop:Fourier-on-idele-class-group}.
Let us specify their values.

\subsubsection{Estimate of errors}\label{sec:choosing-Y-and-M}
First, we want that the term
\begin{align}
    O_K(\totient (\idealq  ) Y^{n-1})O_K(N\Kpseudopolydecay{N})  
\end{align}
in \eqref{eq:Fourier-applied-to-prime-ideals} remain in the form $O_K(N\Kpseudopolydecay{N}) $.
For this, we declare that the $O_K(1)$ constant in \eqref{eq:norm-bound-idq} in the statement of Theorem \ref{thm:main-theorem}
be large enough (note that $\totient (\idealq )< \Nrm (\idealq )$)
and that the parameter $Y$ be an $e^{\sqrt{\log N}/O_K(1)}$ with $O_K(1)$ large enough.
(The threshholds for these choices of course depends on the $O_K(1)$ constant in the exponent in the displayed formula.)

Having fixed $Y$, we take $M$ to be an $e^{\sqrt{\log N}/O_K(1)}$ with $O_K(1)$ even larger so that $H((\pi ;\pi ;0)) = O_n(M\frac{\log Y}{Y})$ will be an $O_K (\Kpseudopolydecay{N} )$;
it follows then that in \eqref{eq:Fourier-applied-to-prime-ideals}
\begin{align}
    \sum _\pi H((\pi ;\pi ;0)) = O_K (N \Kpseudopolydecay{N})
\end{align}
because the number of $\pi $'s is at most 
$\# \{
    x\in \ideala_\lambda \cap \mathcal D  \mid \Nrm (x\ideala_\lambda \inv )<N 
    \}
    = O_K(N)
$.

The above choices allow us to conclude also 
\begin{align}
    \sum _{\pi }G((\pi ;\pi ;0)) = O_K (N\Kpseudopolydecay{N}) .
\end{align}

As a result, for every $\alpha ,P$ as in the beginning of \S \ref{sec:if-P-small-enough} we get:
\begin{multline}\label{eq:small-P-prime-elements}
    \sum _{\substack{
        \pi\in \ideala _\lambda \cap \Rpos P ,\\ \Nrm (\pi \ideala_\lambda\inv )< N, \\ \pi \equiv \alpha \text{ in }  \ideala_\lambda /\idealq  \ideala_\lambda 
        } }
    \log \Nrm (\pi \ideala_\lambda\inv )
    = \\
    \frac{\vol_{\midele } (P\times\{\alpha\})}{\vol_{\midele } (C(\idealq  )/\R \baci_{>0})} 
    \left( N
        -\frac{N^\beta }{\beta }
        \psi _{\mathrm{Siegel}}(P\times \{\alpha \})
    \right) 
    +
    O_K(N\Kpseudopolydecay{N}) .
\end{multline}

\subsection{Slight change of notation}
Let us rewrite the condition $\Nrm (\pi \ideala_\lambda\inv )<N$ in the left hand side
of \eqref{eq:small-P-prime-elements}
as $\Nrm (\pi )< \Nrm (\ideala_\lambda ) N=:N_1$ and write this $N_1$ as $N$ afresh.
Also,
we may consider %
this formula 
for two values $N'<N$
and take the difference. This gives:
\begin{corollary}\label{cor:prime-elements-between-N-and-N'}
    Let $N'<N$.
    Let $\idealq  $ be a non-zero ideal whose norm satisfies 
    \begin{align}
        \Nrm (\idealq  )< \Kpseudopolygrowth{N'}. %
    \end{align}
    Let $\lambda \in \ClK $, $\alpha \in (\ideala_\lambda /\idealq  \ideala _\lambda )\baci $ be arbitrary, and $P\subset (\KR)\baci /\Rpos $ be a convex connected open set small enough (depending on $\totient (\idealq  )$) 
    to satisfy the hypothesis of Proposition \ref{prop:Fourier-on-idele-class-group}.

    Then we have 
    \begin{multline}\label{eq:difference-of-N-and-N'}
        \sum _{\substack{
        \pi\in \ideala _\lambda \cap \Rpos P ,\\ N'<\Nrm (\pi  )< N, \\ \pi \equiv \alpha \text{ in }  \ideala_\lambda /\idealq  \ideala_\lambda 
        } }
    \log \Nrm (\pi \ideala_\lambda\inv )
    = \\
    \frac{1}{\Nrm (\ideala_\lambda )}
    \frac{\vol_{\midele } (P\times\{\alpha\})}{\vol_{\midele } (C(\idealq  )/\R \baci_{>0})} 
    \left(
    (N-N')
    -\frac{\psi _{\mathrm{Siegel}}(P\times \{\alpha \})}{\beta \Nrm(\ideala_\lambda )^{\beta -1} }
    (N^\beta -N^{\prime \beta })
    \right)
    \\ 
    +
    O_K(N\Kpseudopolydecay{N}) .
    \end{multline}

\end{corollary}

\subsubsection{Description as an integral}\label{sec:description-as-integral}
It will be convenient to note that we can describe the previous quantity as an integral
\begin{align}
    (N-N')-\frac{\psi _{\mathrm{Siegel}}(P\times \{\alpha\} )}{\beta \Nrm(\ideala_\lambda )^{\beta -1}}(N^\beta - N^{\prime \beta })
    = \int _{N'}^N 1
    -
    \frac{\psi _{\mathrm{Siegel}}(P\times \{\alpha\}  )}{\Nrm(\ideala_\lambda )^{\beta -1}}
    x^{\beta -1} dx .
\end{align}

Recall the basis $\bm u_1,\dots ,\bm u_{r_1+r_2}$ and the associated coordinates $\xi _1,\dots ,\xi _{r_1+r_2}$ of $\R^{r_1+r_2}$ from \S \ref{sec:measures}.
The norm map $(\KR)\baci \to \R $ looks like 
the composite 
\begin{align}
    (\KR )\baci \cong \R^{r_1+r_2}\times (S^1)^{r_2}\times \{\pm 1\}^{r_1}
    \xrightarrow{\text{projection}} 
    \R^{r_1+r_2}
    &\to \R 
    \\ 
    (\xi _1,\dots ,\xi _{r_1+r_2}) 
    &\mapsto 
    e^{\xi _{r_1+r_2}}.
\end{align}
For the reader's eyeballs' comfort, write $\xi :=\xi _{r_1+r_2}$.
By the substitution $x=e^{\xi }$ we can compute the previous integral (multiplied by the constant factor $\vol_{\mult} (P)$) as 
\begin{multline}
    \vol _{\mult}(P)
    \int _{N'}^N 
    1
    -
    \frac{\psi _{\mathrm{Siegel}}(P\times \{\alpha\}  )}{\Nrm(\ideala_\lambda )^{\beta -1}}
    x^{\beta -1} dx
    \\ 
    = 
    \int\limits _{[\log N',\log N]\times P}
    e^{\xi}
    - 
    \frac{\psi _{\mathrm{Siegel}}(-;\alpha )}{\Nrm(\ideala_\lambda )^{\beta -1}}
    e^{\beta \xi  } 
    \ |d\xi d\xi _1\dots \xi _{r_1+r_2-1}|
    .
\end{multline}
See \eqref{eq:interval-cdot-P} for the definition of the subset 
$[\log N' ,\log N]\times P$
of  
$\R \times H\times (S^1)^{r_2}\times \{\pm 1\} ^{r_1}$.

By \eqref{eq:relation-of-measures} noting $\Nrm (-)=e^\xi $, this can be written as an integral in $d\mu _{\add }$
(let us also switch to the additive notation 
$[N',N]\cdot P\subset \KR $ from \S \ref{sec:decomposition-of-R-r1+r2}):
\begin{align}
    =
    2^{r_2}\sqrt{r_1+r_2}
    \int\limits _{[N',N]\cdot P}
    1-
    \frac{\psi _{\mathrm{Siegel}}(-;\alpha )}{\Nrm(\ideala_\lambda )^{\beta -1}}
    \Nrm(-)^{\beta -1} d\mu _{\add }.
\end{align}

Now we can rewrite the main terms in \eqref{eq:difference-of-N-and-N'}.
By the choice of $\mu _{\midele }$ we have 
$\vol _{\midele }(P\times\{\alpha\})=\vol _{\mult}(P)$.
We also know the value \eqref{eq:volume-of-idele-class-group} of $\vol _{\midele }(C(\idealq  )/\Rpos )$.
We conclude that the main terms of \eqref{eq:difference-of-N-and-N'} are rewritten as
\begin{align}\label{eq:main-terms-rewritten}
    &
    \frac{1}{\Nrm(\ideala_\lambda )}
    \frac{\vol_{\midele } (P)}{\vol_{\midele } (C(\idealq  )/\R \baci_{>0})} 
    \left(
    (N-N')
    -\frac{\psi _{\mathrm{Siegel}}(P\times \{\alpha \})}{\beta \Nrm(\ideala_\lambda )^{\beta -1}}
    (N^\beta -N^{\prime \beta })
    \right)
    \\ 
    =
    &\frac{w_K}{\Nrm(\ideala_\lambda )\totient (\idealq  )2^{r_1}\pi ^{r_2}h_K R_K }
    \int\limits _{[N',N]\cdot P}
    1-
    \psi _{\mathrm{Siegel}}(-;\alpha )
    \left(
        \frac{\Nrm(-)}{\Nrm(\ideala_\lambda )} 
        \right)^{\beta -1} 
    d\mu _{\add }.
\end{align}
Namely: 

\begin{corollary}\label{cor:main-thm-for-Rpos.P}
    Let $1<N'<N$ be positive numbers and $\idealq  $ be an ideal whose norm obeys the bound: 
    \begin{align}
        \Nrm (\idealq  ) < \Kpseudopolygrowth{N'}. %
    \end{align}
    
    Then Theorem \ref{thm:main-theorem} is true for subsets of the form 
    $C= [N',N]\cdot P$
    where $P\subset (\KR)\baci /\Rpos $ is an open set
    which maps injectively into $(\KR)\baci /\Rpos \OK\baci $ and such that the image is a small convex set of size $\ll _n \varphi (\idealq  )^{-(n-2)}$
    in the sense of Definition \ref{def:small-convex-subsets}.
\end{corollary}

Let us repeat that the definition of the set $[N',N]\cdot P$ can be found at \eqref{eq:interval-cdot-P}.

\section{End of proof}\label{sec:end-of-proof}

Lastly, we will ``integrate'' the estimate \eqref{eq:difference-of-N-and-N'} \eqref{eq:main-terms-rewritten} in various small convex subsets $P\subset (\KR)\baci /\Rpos $
to prove 
Theorem \ref{thm:main-theorem}.
Some care must be taken to make sure that the bound of the error will remain uniform in the ideal parameter $\idealq  $.

Let $Y>1$ be again an auxiliary paramater which will be set to be an $e^{\sqrt{\log N}/O_K(1)}$.

\subsection{Reduction to the case of annulus sectors}\label{sec:reduction-to-special-case}
Let the notation be as in Theorem \ref{thm:main-theorem}: in particlular 
let $X>1$ be a positive number.
Write $N:= X^n $. 
Let $C\subset (\KR )_{<X}$ be a convex open set.

Let $\idealq  \subset \OK $ be an ideal whose norm obeys the upper bound \eqref{eq:norm-bound-idq}
and let $\alpha \in (\ideala_\lambda /\idealq \ideala_\lambda )\baci $ for some chosen ideal class $\lambda \in \ClK $.

\subsubsection{Elements close to the axes}\label{sec:points-too-close-to-axes-don't-matter}
The number of elements of $\ideala_\lambda $ in the region 
\begin{align}\label{eq:points-too-close-to-axes}
    \{ x\in (\KR )_{<X} \mid \ 
    &\text{ some of the coordinates } \\ 
    &x=(x_1,\dots ,x_{r_1+r_2})\in \KR \cong \R ^{r_1}\times \mathbf{C}^{r_2} \\ 
    &\text{ has size }
    |x_i| <1  \}
\end{align}
is $O_K(X^{n-1})=O_K(N/X )$.
Hence the contribution of these points to the sum $\sum _{\pi } \log \Nrm (\pi \ideala_\lambda\inv ) \le \sum _{\pi } \log N $ is $O_K(N \frac{\log X }X)$.
Also, the contribution of this region to the integral $\int _C (1-\psi _{\mathrm{Siegel}}(-;\alpha )\Nrm ((-)\ideala_\lambda\inv )d\mu $ is $O_n(N/X)$.

Therefore the contribution from points in \eqref{eq:points-too-close-to-axes} can be ignored.

%
%

\begin{definition}\label{def:annulus-sector}
    An {\it annulus sector} of size $X/Y$ is a connected component of a subset of $\KR $ of the following Baumkuchen-like shape,
    determined by the choice of parameters $a_1,\dots ,a_{r_1+r_2}$, $\theta _1,\dots ,\theta _{r_2}\in \R $:
    \begin{align}
        B=&B(a_1,\dots ,a_{r_1+r_2};\theta _1,\dots ,\theta _{r_2})
        \\ 
            =&
        \{ x \in \KR \mid \ 
        0 < | x |_{\sigma _i} -a_i < X /Y  \ \text{ for }\forall i,
        \\ 
        &0< \arg ( |x|_{\sigma _i } )-\theta _i <1/Y \pmod{2\pi \Z }
        \ \text{ for }r_1+1\le \forall i \le r_1+r_2
        \}
        .
    \end{align}
\end{definition}

We want to approximate the convex body $C$ by the union of disjoint annulus sectors of size $\le X/Y$.
First we cover $(\KR )_{<X}$ with $O_n(Y^n)$ many disjoint annulus sectors $B_j $ of size $\le X/Y$.
Then let 
\begin{align}
    B_{j } \quad (j \in J) 
\end{align}
be those annulus sectors contained in $C$ and not intersecting with the region \eqref{eq:points-too-close-to-axes}.

\begin{lemma}\label{lem:packing-argument}
    We have $\vol (C\setminus \dunion _{j \in J} B_j) \ll _n N/Y$.
    Consequently,
    there are $\ll _K N/Y$ points in $C\setminus \bigsqcup _{j \in J} B_j $. 
\end{lemma}
\begin{proof}
Let $B_j $ ($j\in J'$) be those annulus sectors intersecting with $\partial C$ or the region \eqref{eq:points-too-close-to-axes}.
Since an annulus sector of size $\le X/Y$ has diameter $O_n(X/Y)$, the volume covered by $B_j$'s ($j\in J'$) is
    $O_n(\frac{X}{Y}\cdot \bigl( \area (\partial C)+O_n(X^{n-1}) \bigr) ) $.
    Because $\area (\partial C)=O_n(X^{n-1})$ (see e.g.\ \cite[Lemma A.1]{LinearEquations}),
    this quantity is an 
    $O_n(X^n/Y)=O_n(N/Y)$.
\end{proof}

\subsubsection{Annulus sectors inside $C$}
By Lemma \ref{lem:packing-argument}, we have 
\begin{align}
    \sum _{\substack{
        \pi\in\ideala_\lambda \cap C, \\
        \pi \equiv \alpha \text{ in }\ideala_\lambda /\idealq  \ideala_\lambda 
        }}
    \log \Nrm (\pi\ideala_\lambda\inv )
    = 
    \left( 
        \sum _{j\in J}
    \sum _{\substack{
        \pi\in\ideala_\lambda \cap B_j , \\
        \pi \equiv \alpha \text{ in }\ideala_\lambda /\idealq  \ideala_\lambda 
        }}
    \log \Nrm (\pi\ideala_\lambda\inv )
    \right)
    +
    O_n((N\log N )/Y)  .
\end{align}
Assuming that the conclusion of Theorem \ref{thm:main-theorem} holds for $B_j $, $j\in J$ (with uniform $O_K(1)$ constants of course), 
we can write 
\begin{align}\label{eq:after-packing}
    \sum _{j\in J}
    &\sum _{\substack{
        \pi\in\ideala_\lambda \cap B_j , \\
        \pi \equiv \alpha \text{ in }\ideala_\lambda /\idealq  \ideala_\lambda 
        }}
    \log \Nrm (\pi\ideala_\lambda\inv )
    =\\ 
    &\frac{w_K }{\Nrm (\ideala _\lambda )\totient (\idealq  ) 2^{r_1}\pi^{r_2}h_KR_K }
    \int\limits 
        _{\dunion _{j\in J}
        B_j } 
    1-\psi _{\mathrm{Siegel}} (-;\alpha ) \Nrm ((-)\ideala_\lambda\inv )^{\beta -1}d\mu 
    \\
    &+ \# J \cdot O_{K} (N\Kpseudopolydecay{N}) .
\end{align}
We know $\# J = O_n(Y^n)$
bacause the total number of annulus sectors obeys this bound.
It follows that the last error term in \eqref{eq:after-packing}
is an $O_K(N\Kpseudopolydecay{N})$
if the constant in $Y=e^{\sqrt{\log N}/O_K(1)}$ is taken large enough.

The difference between $\int _{\dunion _{j\in J}B_j }  $ and $\int _C $
is, again by Lemma \ref{lem:packing-argument},
bounded by $O(1)\cdot \# (C\setminus \dunion _{j\in J}B_j )=O_n(N/Y)$.

Thus we are reduced to proving Theorem \ref{thm:main-theorem} for $C=B_j $ an annulus sector of size $\le X/Y$ disjoint from the region \eqref{eq:points-too-close-to-axes}.

\subsection{Proof for annulus sectors}\label{sec:proof-of-special-case}

Now we claim and prove that Theorem \ref{thm:main-theorem} is true for annulus sectors $B\subset \KR _{<X}$ of size $< X/10$ say, disjoint from \eqref{eq:points-too-close-to-axes}.

Let $Y_2>1$ be a new auxiliary integer parameter\footnote{
Here we could have used the letter $Y$ afresh but we prefer to use a different symbol
bacause in the logical order, the $Y$ in \S \ref{sec:reduction-to-special-case} cannot be specified until we specify $Y_2$.
} 
of the form $e^{\sqrt{\log N}/O_K(1)}$.

\subsubsection{Annulus sectors in the multiplicative coordinates}
\label{sec:annulus-sectors-are-cubes-in-some-coordinates}

Via the bijection 
$(\KR )\baci \cong \R^{r_1+r_2}\times (\RZ )^{r_2}\times \{ \pm 1\}^{r_1}$,
our annulus sectors look like cubes contained in $[0,\log X]^{r_1+r_2}\times (\RZ )^{r_2}\times\{\pm 1\}^{r_1}$ 
whose facets are parallel to coordinate hyperplanes.
The left-hand bound ``$0$'' in $[0,\log X]^{r_1+r_2}$ is due to the assumption that our set is disjoint from \eqref{eq:points-too-close-to-axes}.

\subsubsection{Cubes in $H\times (S^1)^{r_2}\times \{\pm 1\}^{r_1}$}\label{sec:cubes} %
We consider the image of $[0,\log X]^{r_1+r_2}$ by the projection 
$\R^{r_1+r_2}=\R\times H \to H$. %
We cover it with cubes of size $ 1/Y_2$ with respect to a fixed $\R$-basis of $H$.
This requires $O_n((Y_2\log X)^{r_1+r_2-1})$ cubes.

We also cover the torus $(\RZ )^{r_2}$ by $Y_2^{r_2}$ cubes of size $ 1/Y_2$.

\begin{definition}\label{def:cube-in-KR-Rpos}
    By a {\it cube} in $(\KR)\baci /\Rpos = H\times (\RZ )^{r_2}\times \{ \pm 1\}^{r_1}$ let us mean a subset $P$ obtained as the product of 
    a cube in $H$ and a cube in $(\RZ )^{r_2}$ of sizes $1/Y_2$ as in \S \ref{sec:cubes}, placed in one of the connected components $\{ \pm 1\}^{r_1}$.
    There are $O_n(Y_2^{n-1}\log ^{r_1+r_2-1} X)$ of them.
\end{definition}

\subsubsection{Sandwiching by special type of subsets}
Let $P$ be a cube in $(\KR)\baci /\Rpos $ in the sense of Definition \ref{def:cube-in-KR-Rpos}
and consider the intersection 
$B\cap \Rpos P $ in $(\KR)\baci $.

For $1<a<b$, recall from \S \ref{sec:decomposition-of-R-r1+r2} that we write $[a,b]\cdot P \subset (\KR)\baci $ for the set of points 
where the projection to $H\times (S^1)^{r_2}\times \{\pm 1\}^{r_1}$ belong to $P$
and the value of $\Nrm (-)$ is in $[a,b]$. 

By \S \ref{sec:points-too-close-to-axes-don't-matter}, up to negligible error  
there are two real numbers $\log (N/X)< a<b<\log N$ satisfying 
\begin{align}
    [a,b]\times P \subset B\cap (\R \times P) \subset [a-\varepsilon ,b+\varepsilon ]\times P
    \quad \text{ in } \R \times H\times (S^1)\times \{\pm 1\}^{r_1},
\end{align}
where $\varepsilon = O_n(1 / Y_2)$ by \S \ref{sec:annulus-sectors-are-cubes-in-some-coordinates}.
In the additive notation, this means 
\begin{align}
    [e^a,e^b]\cdot P \subset B\cap \Rpos P \subset [e^a(1-\varepsilon ), e^b(1+\varepsilon )]\cdot P
    \quad \text{ in }\KR =\R^{r_1}\times \mathbf{C} ^{r_2}.
\end{align}

It follows that 
\begin{align}\label{eq:sum-le-sum-le-sum}
    \sum _{\substack{
        \pi \in \ideala_\lambda \cap ([e^a,e^b]\cdot P)
         \\ 
        \pi \equiv \alpha \text{ in }\ideala_\lambda /\idealq  \ideala_\lambda } }
    \log \Nrm (\pi\ideala_\lambda\inv )
    &\le 
    \sum _{\substack{
        \pi \in \ideala_\lambda \cap (B\cap \Rpos P )
         \\ 
        \pi \equiv \alpha \text{ in }\ideala_\lambda /\idealq  \ideala_\lambda  } }
    \log \Nrm (\pi\ideala_\lambda\inv )
    \\ 
    &\le 
    \sum _{\substack{
        \pi \in \ideala_\lambda \cap ([e^a(1-\varepsilon ),e^b(1+\varepsilon )]\cdot P )
         \\ 
        \pi \equiv \alpha \text{ in }\ideala_\lambda /\idealq  \ideala_\lambda } }
    \log \Nrm (\pi\ideala_\lambda\inv ) .
\end{align}

\subsubsection{Applying the main theorem for the special case}\label{sec:choosing-Y2-and-bound-q}
We want to apply Corollary \ref{cor:main-thm-for-Rpos.P}
to the sets $[e^a,e^b]\cdot P$ and $[e^a(1-\varepsilon ),e^b(1+\varepsilon )]\cdot P$
to rewrite \eqref{eq:sum-le-sum-le-sum}.
We have to verify the hypotheses on $\Nrm (\idealq  )$ and the size of $P$:
\begin{itemize}
    \item We know $e^{\sqrt{\log (e^a(1-\varepsilon ))}} > e^{\sqrt{\log (N/X)-\varepsilon }}>e^{\sqrt{\log N}/2}>\Nrm (\idealq  )$ by \eqref{eq:norm-bound-idq};
    \item By the construction of $P$, the image of $P\times \{\alpha \}$ to $(\KR)\baci /\Rpos $ is a cube of size $\ll _K 1/Y_2$.
    By declaring that we shall take the $O_K(1)$ in the bound \eqref{eq:norm-bound-idq} of $\Nrm (\idealq  )$ large enough depending on that in $Y_2$, we can ensure that 
    the image of $P\times \{\alpha \}$ is small enough (and automatically convex) to satisfy 
    $1/Y_2 \ll _K \Nrm (\idealq  )^{-(n-2)}< \totient (\idealq  )^{-(n-2)}$ as required in Proposition \ref{prop:Fourier-on-idele-class-group}.
\end{itemize}
Under these circumstances we can apply Corollary \ref{cor:main-thm-for-Rpos.P} and obtain an estimate of our sum from below and above:
\begin{align}\label{eq:sandwich}
    \frac{w_K }{\Nrm (\ideala_\lambda )\totient (\idealq )2^{r_1}\pi^{r_2} h_KR_K}
        &\int\limits _{[e^a,e^b]\cdot P} 1- \psi _{\mathrm{Siegel}}(-;\alpha ) \Nrm ((-)\ideala_\lambda \inv )^{\beta -1} 
        d\mu _{\add}
        \\[10pt] 
        &+ O_{K}(N\Kpseudopolydecay{N})
        \\ 
        < \sum _{\substack{
            \pi \in \ideala_\lambda \cap (B\cap \Rpos P ) 
             \\ 
            \pi \equiv \alpha \text{ in }\ideala_\lambda /\idealq  \ideala_\lambda } }
        \log \Nrm (\pi\ideala_\lambda\inv )
        &
        \\ 
        < 
        \frac{w_K }{\Nrm (\ideala_\lambda )\totient (\idealq )2^{r_1}\pi^{r_2} h_KR_K}
        &\int\limits _{[e^a(1-\varepsilon ),e^b(1+\varepsilon )]\cdot P} 1- \psi _{\mathrm{Siegel}}(-;\alpha ) \Nrm ((-)\ideala\inv )^{\beta -1} 
        d\mu _{\add}
        \\
        &+ O_{K}(N \Kpseudopolydecay{N})
        .
\end{align}

\subsubsection{Difference of upper and lower bounds}
The difference of the main terms of the upper and lower bounds in \eqref{eq:sandwich}
is bounded by
\begin{multline}\label{eq:difference-of-the-integrals}
    \frac{w_K }{\Nrm (\ideala_\lambda )\totient (\idealq )2^{r_1}\pi^{r_2} h_KR_K} \cdot 
    \\ 
        \int\limits _{([e^a(1-\varepsilon ),e^b(1+\varepsilon )]\cdot P) \setminus ([e^a,e^b ]\cdot P) }
        \Bigl| 
            1- \psi _{\mathrm{Siegel}}(-;\alpha ) \Nrm ((-)\ideala\inv )^{\beta -1} 
        \Bigr|
            d\mu _{\add}
\end{multline}
and is an $O_{K}(N / (Y_2)^n)$ because:
the factor in front of the integral is an $O_K(1)$,
the integrand always has absolute value $\le 2$ and 
the domain of integration has volume $< 2N\varepsilon \area (P)\ll N/(Y_2)^{n}$.
By \eqref{eq:sandwich} and the triangle inequality we conclude 
\begin{multline}\label{eq:conclusion-for-each-P}
    \sum _{\substack{
            \pi \in \ideala_\lambda \cap (B\cap \Rpos P ) 
             \\ 
            \pi \equiv \alpha \text{ in }\ideala_\lambda /\idealq  \ideala_\lambda } }
        \log \Nrm (\pi\ideala_\lambda\inv )
        =
        \\ 
        \frac{w_K }{\Nrm (\ideala_\lambda )\totient (\idealq )2^{r_1}\pi^{r_2} h_KR_K}
        \int\limits _{B\cap \Rpos P} 
        1- \psi _{\mathrm{Siegel}}(-;\alpha ) \Nrm ((-)\ideala_\lambda \inv )^{\beta -1} 
        d\mu _{\add}
        \\[10pt] 
        + O_{K}(N/(Y_2)^{n})+O_K(N\Kpseudopolydecay{N}).
\end{multline}

\subsubsection{To take the sum over the cubes}
Take the sum $\sum _P$ of \eqref{eq:conclusion-for-each-P} over the cubes $P$ in the sense of Definition \ref{def:cube-in-KR-Rpos}.
Since there are $O_K((Y_2)^{n-1}\log ^{r_1+r_2-1}X)$ of them, we obtain 
\begin{multline}\label{eq:two-error-terms}
    \sum _{\substack{
            \pi \in \ideala_\lambda \cap B  
             \\ 
            \pi \equiv \alpha \text{ in }\ideala_\lambda /\idealq  \ideala_\lambda } }
        \log \Nrm (\pi\ideala_\lambda\inv )
        =
        \\ 
        \frac{w_K }{\Nrm (\ideala_\lambda )\totient (\idealq )2^{r_1}\pi^{r_2} h_KR_K}
        \int _{B} 
        1- \psi _{\mathrm{Siegel}}(-;\alpha ) \Nrm ((-)\ideala_\lambda \inv )^{\beta -1} 
        d\mu _{\add}
        \\[10pt]
        + O_{K}\left( N\frac{\log ^{r_1+r_2-1}X}{Y_2}\right)
        \\ 
        +O_K((Y_2)^{n-1}\log ^{r_1+r_2-1}X \cdot N\Kpseudopolydecay{N} ).
\end{multline}
The first error term in \eqref{eq:two-error-terms} is an $O_K(N\Kpseudopolydecay{N})$ with the $O_K(1)$ in the exponent being more or less the same as that in $Y_2=e^{\sqrt{\log N}/O_K(1)}$.

By taking the $O_K(1)$ constant in $Y_2$ large enough compared to the constant in $\Kpseudopolydecay{N}$ in the second error term,
we can make the second error term also an 
$O_K(N\Kpseudopolydecay{N})$.

If follows that if we set the $O_K(1)$ constant in $Y_2$ large enough, the two error terms in \eqref{eq:two-error-terms} become a new $O_K(N\Kpseudopolydecay{N})$.
Having fixed $Y_2$, we can now finalize the $O_K(1)$ constant in the restriction \eqref{eq:norm-bound-idq} on $\Nrm (\idealq )$ following the requirements from \S\S \ref{sec:choosing-Y-and-M}, \ref{sec:choosing-Y2-and-bound-q}.
This proves Theorem \ref{thm:main-theorem} for annulus sectors $B$ and accomplishes the goal of \S \ref{sec:proof-of-special-case}.

Then we run the process of \S \ref{sec:reduction-to-special-case} 
to fix the $O_K(1)$ constant in $Y=e^{\sqrt{\log N}/O_K(1)}$ and deduce Theorem \ref{thm:main-theorem} in general.
This completes the proof of Theorem \ref{thm:main-theorem}.

\appendix

\section{Conventions on $\KR $}\label{sec:Dirichlet}

We have to fix some conventions in order to be able to do analysis on $(\KR )\baci $ and its quotients
$(\KR)\baci /\OK \baci $ and $(\KR )\baci /\OK\baci \Rpos $.
Following the usual convention, we number the embeddings as
\begin{align}
    \sigma _1,\dots ,\sigma _{r_1};\quad \sigma _{r_1+1},\dots ,\sigma _{r_1+r_2}; \quad \sigma _{r_1+r_2+1},\dots ,\sigma _{r_1+2r_2}
\end{align}
where $\sigma _1,\dots ,\sigma _{r_1}$ are the different real embeddings and $\sigma _{r_1+1},\dots ,\sigma _{r_1+r_2}$ represent the different conjugate pairs of imaginary embeddings.

\subsection{Norm functions}\label{sec:norm}
For a non-zero ideal $\ideala \subset \OK $, we write 
$\Nrm (\ideala ):= \# (\OK /\ideala )$ for its norm.
For a non-zero element $\alpha \in \OK $, by slight abuse of notation we write 
$\Nrm (\alpha ):= \Nrm (\alpha \OK ) = |N_{K/\Q}(\alpha )|$.
The norm functions $\Nrm (-)$ and $N_{K/\Q }$ extend to $\KR $ naturally,
\begin{align}
    \xymatrix@C=100pt{
        K \ar[r]^{\Nrm \text{ or }N_{K/\Q }}\ar@{_{(}->}[d] 
        &\Q \ar@{_{(}->}[d] 
        \\ 
        \KR \ar[r]^{\Nrm \text{ or }N_{K /\Q }} 
        & \R ,
    }
\end{align}
where we have denoted the extensions by the same symbols.

Let us emphasize that $\Nrm (-)$ always takes positive values; 
not to be confused with the ring-theoretic norm $N_{K/\Q}(-)$, which does take negative values unless $K$ is totally imaginary.

\subsection{The group structure of $(K\otimes _\Q \R)\baci $}\label{sec:Minkowski-iso}
Let $\sigma _i \otimes _\Q \R \colon \KR \to \R $ ($1\le i\le r_1$) and $\sigma _i\otimes _\Q \R \colon \KR \to \mathbf{C} $ ($r_1 + 1\le i\le r_1+r_2$) be the unique $\R $-linear extentions.
The map they jointly define 
\begin{align}\label{eq:Minkowski-iso}
    \KR \isoto \R ^{r_1}\times \mathbf{C} ^{r_2}
\end{align}
is known to be an isomorphism of $\R$-algebras. %

It follows that 
\begin{align}\label{eq:iso-KRbaci}
    (\KR)\baci
    \ \cong\ 
    (\R\baci)^{r_1}\times (\mathbf{C}\baci)^{r_2}
    \ \cong\  
    \R ^{r_1+r_2}\times (S^1)^{r_2}\times \{\pm 1\}^{r_1},
\end{align}
where the second isomorphism 
follows the conventions in 
\S \ref{sec:multiplicative-groups};
in particular the map $\mathbf{C}\baci \to \R $ is defined by $z\mapsto \log (|z|^2)$.

\subsection{Decomposition of $\R^{r_1+r_2}$}\label{sec:decomposition-of-R-r1+r2}
Write 
\begin{align}
    \bm u_{r_1+r_2}:= \frac{1}{n}
    (\numberOfTerms{1,\dots,1}{r_1}, 
    \numberOfTerms{2,\dots ,2}{r_2} ) 
    \in \R^{r_1+r_2}.
\end{align}
In \eqref{eq:iso-KRbaci}, the diagonal image $\Rpos \subset (\R\baci)^{r_1}\times (\mathbf{C}\baci)^{r_2}$ corresponds to the subspace 
\begin{align}
    \R \cdot \bm u_{r_1+r_2}
    \times 
    \{1\}^{r_2}
    \times 
    \{+1\}^{r_1}
    \quad\subset\quad 
    \R ^{r_1+r_2}\times (S^1)^{r_2}\times \{\pm 1\}^{r_1} .
\end{align}
Let $H\subset \R^{r_1+r_2}$ be the following $\R $-linear subspace of codimension $1$:
\begin{align}
    H:= \left\{ (x_1,\dots ,x_{r_1+r_2}) \ \middle| \ 
    \sum _{i=1}^{r_1+r_2} x_i =0 \right\} .
\end{align}
This corresponds to 
elements
$x=(x_1,\dots ,x_{r_1};
z_1,\dots ,z_{r_2}) \in (\R\baci )^{r_1}\times (\mathbf{C}\baci) ^{r_2}$
of norm $\pm 1$.

We have the direct sum decomposition 
\begin{align}
    \R ^{r_1+r_2}
    = \R \cdot \bm u_{r_1+r_2}
    \oplus 
    H
\end{align}
and hence the isomorphism
\begin{align}
    (\KR)\baci /\Rpos \cong H\times (S^1)^{r_2}\times \{\pm 1\}^{r_1}.
\end{align}
This also gives a decomposition 
\begin{align}
    (\KR )\baci &=\R \times H\times (S^1)^{r_2}\times \{\pm 1\}^{r_1} 
    \\ 
    &= \Rpos \times (\KR)\baci /\Rpos .
\end{align}

Given a subset $P\subset (\KR )\baci/\Rpos $, we write $\Rpos \cdot P$ for its inverse image in $(\KR)\baci $.
When we see it as a subset of $\R \times H\times (S^1)^{r_2}\times \{\pm 1\}^{r_1}$,
we prefer to use the symbol $\R \times P$.

If furthermore $0<a<b$ are given, we write 
\begin{align}\label{eq:interval-cdot-P}
    [a,b]\cdot P := \{ x\in \Rpos \cdot P \mid a\le \Nrm (x)\le b \} .
\end{align}
When we see it as a subset of $\R \times H\times (S^1)^{r_2}\times \{ \pm 1\}^{r_1}$,
we write it as $[\log a,\log b]\times P$.

\subsection{The domain $(\KR)_{<X}$}\label{sec:the_domain}
For an embedding $\sigma \colon K\inj \mathbf{C} $ and $x\in \KR $, we set 
\begin{align}
    |x|_\sigma := |(\sigma \otimes _\Q \R )(x)|,
\end{align}
where $|-| $ in the right hand side is the usual absolute value in $\mathbf{C} $.
For a positive number $X>1$, we define a subset $(\KR )_{<X} \subset \KR $ by 
\begin{align}
    (\KR )_{<X} 
    :=
    \{ x\in \KR \mid |x|_{\sigma _{i}} < X \text{ for all }1\le i\le r_1+r_2 \} .
\end{align}

\subsection{Measures}\label{sec:measures}
We will consider two measures on $(\KR)\baci $.
One is the restriction of the Lebesgue measure by the inclusion
$(\KR)\baci \subset \KR \cong \R^{r_1}\times \mathbf{C} ^{r_2}\cong \R^{n}$,
where we took the isomorphism $\mathbf{C} \cong \R^2$; $x+iy \leftrightarrow (x,y)$.
We denote it by $\mu _{\add }$.

The other is obtained through 
$(\KR)\baci \cong \R^{r_1+r_2}\times (S^1)^{r_2}\times \{ \pm 1\} ^{r_1}$.
We give $\R^{r_1+r_2}$ the Lebesgue measure and $(S^1)^{r_2}$ the Haar measure such that the volume of $(S^1)^{r_2}$ is $(2\pi )^{r_2}$. (In other words, as far as volume is concerned, we regard $S^1$ as $\R / 2\pi \Z$, not $\RZ$.)
Each point of $\{\pm 1\}^{r_1}$ is given weight $1$.
We denote the resulting product measure on $(\KR)\baci $ by $\mu_{\mult}$.

By some calculus one finds
\begin{align}
    d\mu _{\mult } = \frac{2^{r_2}}{\Nrm (-)} d\mu _{\add } \quad \text{ on }(\KR)\baci , 
\end{align}
where $\Nrm (-)$ is the norm function in \S \ref{sec:norm}.

    We endow $H\subset \R^{r_1+r_2}$ with the induced metric and hence the measure $\mu _H$.
    This endows $(\KR )\baci /\Rpos =H\times (S^1)^{r_2}\times \{\pm 1\}^{r_1}$ with the product measure. We denote it by $\mu _{\mult}$ as well.

    Let $\bm u_{1},\dots ,\bm u_{r_1+r_2-1}\in H$ be any orthonormal basis 
    and $(\xi _1,\dots ,\xi _{r_1+r_2})$ be the coordinate system corresponding to the basis $\bm u_{1},\dots ,\bm u_{r_1+r_2}$ of $\R ^{r_1+r_2}$.
    (See \S \ref{sec:decomposition-of-R-r1+r2} for the vector $\bm u_{r_1+r_2}$.)
    Note that the measure $\mu_H$ on $H$ is exactly the one associated with the volume form $d\xi _1\dots d\xi _{r_1+r_2-1}$ on $H$. 
    The inner product of $\bm u_{r_1+r_2}$ and the unit normal vector $(1,\dots ,1)/\sqrt{r_1+r_2}$ of $H$ is $1/\sqrt{r_1+r_2}$.
    This implies the following relation of measures:
    \begin{align}\label{eq:relation-of-measures}
        |d\xi _1 \dots d\xi_{r_1+r_2}| = \sqrt{r_1+r_2}\, d\mu _{\mult}
        =
        \frac{2^{r_2}\sqrt{r_1+r_2}}{\Nrm (-)} d\mu _{\add}.
    \end{align}

\subsection{The regulator}\label{sec:fundamental-domain-KR}

Dirichlet's Unit Theorem says that we have an isomorphism as follows, which we fix:
    \begin{align}\label{eq:known-structure}
        \OK\baci \cong (\Z / w_K \Z )\times \Z ^{r_1+r_2-1}.
    \end{align}
By its proof, we know that the inclusion map followed by the projection 
\begin{align}
    \OK\baci \inj (\KR )\baci 
    \cong 
    \R \times 
        H \times 
        (S^1)^{r_2}
        \times 
        \{\pm 1\}^{r_1} 
    \surj 
    H
\end{align}
has image which is a full-rank lattice.

The {\it regulator} $R_K>0$ of $K$ is the real number satisfying 
$\vol _{\mult}(H/\Z ^{r_1+r_2-1}) = R_K\sqrt{r_1+r_2}$.

\subsection{The structure of $(\KR)\baci /\OK\baci $}\label{sec:we-take-this-opportunity}
We take this opportunity to compute the abelian group $(\KR)\baci/\OK\baci $ further.
By \S \ref{sec:fundamental-domain-KR} we know it decomposes as 
\begin{align}
    \R \times \frac{(\KR )\baci}{\OK\baci \cdot \Rpos }
    =\R \times  \frac{H\times \{\pm 1\}^{r_1}\times (S^1)^{r_2}}{  \OK\baci }
    .
\end{align}
We have the next commutative diagram with exact rows 
\begin{align}
    \xymatrix{
        1\ar[r] 
        &\Z /w_K\Z \ar[r] \ar[d]
        &\OK\baci \ar[r] \ar[d]
        &\Z ^{r_1+r_2-1}\ar[r] \ar[d]
        &1
        \\
        1\ar[r] 
        &\{\pm 1\}^{r_1}\times (S^1)^{r_2}\ar[r] 
        & H \times \{\pm 1\}^{r_1}\times (S^1)^{r_2} \ar[r] 
        &H \ar[r]
        &0
    }
\end{align}
    This gives an exact sequence
    \begin{align}\label{eq:KR-mod-OK.R-is-torus}
        1\to \frac{\{\pm 1\}^{r_1}\times (S^1)^{r_2}}{ (\OK\baci)_{\tors} }
        \to \frac{(\KR )\baci}{\OK\baci \cdot \Rpos}
        \to (\RZ )^{r_1+r_2-1}\to 0.
    \end{align}
    The left-hand term is isomorphic to 
    $\{ \pm 1\}^{r_1-1}\times (S^1)^{r_2}$ if $r_1\neq 0$
    and 
    $(S^1)^{r_2}$ if $r_1=0$.

    \subsection{The neutral component of $(\KR)\baci /\OK\baci \Rpos $}\label{sec:KR-mod-OK.R-is-torus}
    From \eqref{eq:KR-mod-OK.R-is-torus} it is now clear that 
    $(\KR )\baci / (\OK\baci \Rpos )$
    is a torus (= commutative compact Lie group) of dimension $n-1$.
    As such, its neutral component 
    $((\KR )\baci / (\OK\baci \Rpos ))^0$ is isomorphic to $(S^1)^{n-1}$.
    Choose and fix any such isomorphism 
    \begin{align}\label{eq:neutral-component-of-KR}
        \left( \frac{(\KR)\baci}{\OK\baci \Rpos}
        \right)^0
        \cong (S^1)^{n-1} .
    \end{align}
    Of course there are natural ways of specifying such an isomorphism in the spirit of  \eqref{eq:KR-mod-OK.R-is-torus}
    relying only on mild choices. 
    But this is not needed in this paper.

    \subsection{The volume of $(\KR)\baci /\OK\baci \Rpos $}\label{sec:volume-of-KR-mod-OK.R}
    Since $\OK\baci $ is a discrete subgroup of %
    $(\KR)\baci /\Rpos $, and translation by an element of $\OK\baci $ preserves the measure $\mu _{\mult}$ (see \S \ref{sec:measures} for this measure), the quotient inherits the measure.
    Let us write the resulting one also by $\mu _{\mult}$. 

    It will be convenient to know the volume of $(\KR )\baci /\OK\baci \Rpos $.
    The left-hand term of \eqref{eq:KR-mod-OK.R-is-torus} has volume $2^{r_1}(2\pi )^{r_2}/w_K$.
    The right-hand term has volume $R_K \sqrt{r_1+r_2}$ (\S \ref{sec:fundamental-domain-KR}).
    By the definition of $\mu _{\mult}$ as the product measure, it follows that 
    \begin{align}\label{eq:volume-of-KR-mod-OK.R}
        \vol _{\mult} ((\KR)\baci /\OK\baci \Rpos ) = \frac{2^{r_1}(2\pi )^{r_2}R_K \sqrt{r_1+r_2}}{w_K}
    .\end{align}

    \section{Fourier analysis on a torus}\label{sec:Fourier}

    By a {\it torus} we mean a compact {\it commutative} Lie group.
    By so-called Lie's three theorems, 
    such a group $\T $ is known %
    to be isomorphic to a product 
    \begin{align}\label{eq:torus-is-a-product}
        \T \cong G\times (S^1)^d
    \end{align}
    of a finite {\it commutative} group $G$ and a connected torus $(S^1)^d$.
    Of course the isomorphism 
    is usually non-canonical.

    \subsection{Approximation of indicator functions of small convex subsets by characters}
    Let $\T $ be a torus and 
    \begin{align}\label{eq:def-Pontryagin-dual}
        \hatT := \Hom _{\text{cont}}(\T ,S^1)\cong \widehat G \times \Z ^d 
    \end{align}
    be its Pontryagin dual.
    Every $\xi \in \hatT$ will be seen as a $\mathbf{C} $-valued function on $\T $
    by $\T \xrightarrow{\xi }S^1 \subset \mathbf{C} $.
    Fourier analysis on $\T $ roughly says that functions on $\T $ can be written as $\mathbf{C} $-linear combinations of elements of $\hatT $.
    The precise statement we use in this paper is Proposition \ref{prop:fact-from-Fourier} below.
    
    \begin{definition}\label{def:small-convex-subsets}
        Let $\T $ be a torus whose neutral component $\T ^0$ is equipped with an isomorphism $\T ^0\cong (S^1)^d$.
        A {\it small convex subset} of $\T $ (resp.\ of size $\le \delta $, $0<\delta <1/4$) is a connected subset $P\subset \T $ such that for some $p\in P$ the translate 
        \begin{align}
            P-p \subset \T ^0 \cong (S^1)^d\cong (\RZ )^d
        \end{align}
        is the projection of an open convex subset $\widetilde P\subset [-\frac{1} 4 , \frac{1} 4]^d \subset \R ^d$ (resp.\ $\widetilde P \subset [-\delta ,\delta ]^d$). 
    \end{definition}
    
    \begin{proposition}\label{prop:fact-from-Fourier}
        Let $M,Y>1$ be positive parameters. Let $P$ be a small convex subset of a torus of size $<1/5$ with respect to some identification $\T ^0 \cong (S^1)^d= \R ^d /\Z ^d$.
        Then the indicator function $1_P$ can be approximated as:
        \begin{align}\label{eq:fact-from-Fourier}
            1_P &= \sum _{\xi \in \Xi }
            c_\xi \cdot \xi 
            + G + H ,
        \end{align}
        where 
        \begin{itemize}
        \item $\Xi \subset \hatT $ is a subset of size $O_d( \#\pi_0(\T )\cdot Y^d)$,
        \item the coefficients $c_\xi 
        \in \mathbf{C} $ obey bounds $|c_\xi |\le 1$, 
        \item the values of $G$ have absolute values $\le 1$ and there is another small convex set $P'$ satisfying $\oname{Supp}(G)\subset P'\setminus P$ and $\oname{Vol}(P'\setminus P)=O_n(1/M)$,
        \item $H$ 
        is a function satisfying $\lnorm H _\infty 
        \le O_d(M\cdot \frac{\log Y}{Y})$. 
        \end{itemize}
    
        Moreover, for $\xi\in \left( \pi _0(\T )\right) \widehat{\ } \subset \hatT $ the coefficient $c_{\xi }$ can be explicitly described:
        letting $\mu $ be any Haar measure on $\T $, 
        \begin{align}
            c_\xi 
            = \frac{\mu (P)}{\mu (\T )} \ol{\xi (P)}
            +O_d(1 /M),
        \end{align}
        where the value $\xi (P)\in \mathbf{C} $ is well defined because $\xi \in \left( \pi _0(\T )\right) \widehat{\ }$ is constant on every connected set.
    \end{proposition}
    \begin{proof}
        This follows from \cite[Corollary A.3]{LinearEquations} and \cite[Lemma A.9]{QuadraticUniformity}.
    
        Note that although \cite[Lemma A.9]{QuadraticUniformity} is stated for the connected torus $(\RZ )^d$, the corresponding statement for a non-connected commutative torus easily follows from the decompositions \eqref{eq:torus-is-a-product} \eqref{eq:def-Pontryagin-dual}.
        The specific choice of these decompositions affect 
        the resulting objects $\Xi ,G, H,\dots $ but 
        do not affect 
        the bounds of the error terms.
    \end{proof}
    
    In our application, the number $\# \pi_0(\T )$ of connected components of $\T $ will be allowed to grow at a pseudopolynomial rate $e^{\sqrt{\log N} / O_K(1)}$ at the fastest.
    
    \section{Id{\`e}le class group as a torus}\label{sec:idele-class-group-as-a-torus}

    
    Let $\idealq = \prod _{\idealp } \idealp ^{n_\idealp }$ be a non-zero ideal 
    and let $C(\idealq )$ be the mod $\idealq $ \idele class group from \S \ref{sec:Hecke}.
    We perform
    Fourier analysis from \S \ref{sec:Fourier}  on the torus $C(\idealq )/\R\baci_{>0}$. 
    Its volume \eqref{eq:volume-of-idele-class-group} explains the coefficient in Mitsui's Prime Number Theorem \ref{thm:Mitsui}, \ref{thm:main-theorem}.

    \subsection{The structure of $C(\idealq  )/\Rpos $}
    In this subsection we review how the quotient $C(\idealq )/\R\baci_{>0}$ is a real torus and fix some notation on the way.
    
    \subsubsection{To factor out $\ClK $}
    Recall our notation
    \begin{align}
        C(\idealq )=\widetilde C(\idealq )/K\baci 
        = \left((\KR)\baci \times 
        \bigoplus _{\idealp | \idealq }
        K\baci / (1+\idealp ^{n_{\idealp}} \calO _{K,\idealp } )
        \times \dsum _{\idealp \notdivide \idealq  } \Z \right) / K\baci.
    \end{align}
    Consider the $\idealp $-adic valuation map 
    $v_\idealp \colon K\baci / (1+\idealp ^{n_{\idealp}} \calO _{K,\idealp } ) \surj \Z $
    for each $\idealp \divides \idealq  $. 
    Its kernel is identified with $\calO _{K,\idealp }\baci/ (1+\idealp ^{n_\idealp }\calO _{K,\idealp }) = (\calO _{K,\idealp } / \idealp ^{n_\idealp }\calO _{K,\idealp })\baci = (\calO _{K} / \idealp ^{n_\idealp })\baci$.
    If we take the direct sum over $\idealp \divides \idealq  $, by Chinese Remainder Theorem we get $(\OK / \idealq  )\baci$.
    
    It follows that we have the following commutative diagram with exact rows and injective vertical maps, where $P_K$ denotes the group of principal fractional ideals
    \begin{align}
        \xymatrix{
            1 \ar[r]
            &\OK \baci \ar[d]\ar[r]
            &K\baci \ar[d]\ar[r]
            &P_K \ar[d]\ar[r]
            &1
            \\
            1 \ar[r]
            &\frac{(\KR)\baci}{\Rpos} \times (\OK /\idealq )\baci \ar[r]
            &\widetilde C(\idealq  )/\Rpos \ar[r]_{(v_\idealp )_{\idealp }} \ar[r]
            &\dsum\limits _{\substack{
                \idealp \text{ non-zero} \\  
                \text{prime ideals}
                }} 
                \Z
            \ar[r]
            &0.
        }
    \end{align}
    By the snake lemma we obtain an exact sequence
    \begin{align}\label{eq:decomp-idele-class-group}
        1\to \frac{ \frac{(\KR)\baci}{\Rpos} \times (\OK /\idealq )\baci}{\OK\baci}\to C(\idealq  ) \to \ClK \to 0.
    \end{align}
    
    \subsubsection{To factor out $(\OK/\idealq  )\baci $}
    We want to compute the first term of \eqref{eq:decomp-idele-class-group} further.
    Consider the commutative diagram with exact rows:
    \begin{align}\label{eq:Now-consider}
        \xymatrix{
            &1\ar[r] \ar[d]
            &\OK\baci \ar[r]^{=} \ar[d]
            &\OK \baci\ar[r] \ar[d]
            &1
            \\ 
            1\ar[r] 
            &(\OK /\idealq  )\baci\ar[r] 
            &\frac{(\KR)\baci}{\Rpos} \times (\OK /\idealq )\baci \ar[r] 
            &\frac{(\KR)\baci}{\Rpos } \ar[r] 
            &0 .
        }
    \end{align}
    As the vertical maps are injective, 
    we conclude that $( \frac{(\KR)\baci}{\Rpos} \times (\OK /\idealq )\baci)/ \OK\baci$ is an extension 
    of $\frac{(\KR)\baci}{\Rpos\cdot\OK\baci} $ by $(\OK /\idealq  )\baci$,
    the first of which has been computed in \S \ref{sec:we-take-this-opportunity}.

    \subsubsection{The structure of $C(\idealq  )/\Rpos $---conclusion}
    Our computation so far expresses $C(\idealq  )/\Rpos $ as a twofold extension
    \begin{align}\label{eq:twofold-extension}
        \xymatrix{
            &1 \ar[d] && 
            \\ 
            &(\OK /\idealq  )\baci\ar[d]&&&
            \\
            1\ar[r]
            &( \frac{(\KR)\baci}{\Rpos} \times (\OK /\idealq )\baci)/ \OK\baci \ar[r]\ar[d] 
            & C(\idealq ) / \Rpos \ar[r] 
            &\ClK \ar[r]
            & 0.
            \\ 
            &\frac{(\KR)\baci}{\OK\baci\cdot\Rpos }\ar[d]&&&
            \\
            &1&&&
        }
    \end{align}
    In \S \ref{sec:we-take-this-opportunity} we saw that $\frac{(\KR)\baci}{\OK\baci\cdot\Rpos }$ is a torus.
    It follows that $C(\idealq  )/\Rpos $ is also a torus.
    
    Its neutral component 
    $(C(\idealq  )/\Rpos )^0$
    surjects to the neutral component of 
    $\frac{(\KR)\baci}{\OK\baci\cdot\Rpos }$,
    and the size of the kernel of this map 
    is bounded from above by $\# (\OK /\idealq  )\baci = \totient (\idealq  ) < \Nrm (\idealq  )$.

    \subsubsection{The volume of $C(\idealq  )/\Rpos $}\label{sec:volume-of-idele-class-group}
    Via the diagram \eqref{eq:twofold-extension},
    the volume form of $\frac{(\KR)\baci}{\OK\baci\cdot\Rpos }$ induces 
    a volume form on $C(\idealq  )/\Rpos $.
    Write the resulting measure as $\mu_{\midele }$ and $\oname{Vol}_{\midele}(-)$ the volume of subsets.\footnote{
        Note that this may be different from the measure adopted in the reader's favorite reference because 
        every reference has its own convention on the choice of measures.
        Don't worry, any two differ only by a constant factor because all are Haar measures. 
        }
    By the diagram and \eqref{eq:volume-of-KR-mod-OK.R} we have 
    \begin{align}\label{eq:volume-of-idele-class-group}
        &\vol _{\midele } (C(\idealq  )/\Rpos )
        \\ 
        = &h_K \totient (\idealq  ) \vol _{\mult}(\frac{(\KR)\baci}{\OK\baci\Rpos })
        \\ 
        =& \totient (\idealq  )
        2^{r_1}(2\pi )^{r_2} R_K h_K \sqrt{r_1+r_2}/w_K
        .
    \end{align}

    \subsection{Indicator function of convex sets}
    
    Now we combine Fourier analysis from \S \ref{sec:Fourier} and the description of $C(\idealq )/\Rpos $.
    The consequence, Proposition \ref{prop:Fourier-on-idele-class-group},
    is used when we compute weighted sum of prime elements in {\em thin cones}
    as in Figure \ref{fig:thin-cone} in Introduction.

    \subsubsection{A trivialization of the neutral component}\label{sec:trivialization}
    
    Consider the covering map induced on the neutral components $(\mathbf{C}(\idealq  ) /\R _{>0}\baci)^0 \surj \left( \frac{(\KR)\baci}{\OK\baci\cdot\Rpos } \right)^0$.
    In \eqref{eq:neutral-component-of-KR} we fixed an isomorphism between the target and 
    $(S^1)^{n-1}=(\RZ )^{n-1}$. 
    This identifies the covering map with a map of the form
    \begin{align}\label{eq:cover-is-iso-to}
        \R ^{n-1}/\Lambda (\idealq  )\surj \R ^{n-1}/\Z ^{n-1} 
    \end{align} 
    for some sublattice $\Lambda (\idealq  )\subset \Z ^{n-1}$
    of index dividing $\totient (\idealq  )$.
    By $\Z $-linear algebra (Euclid's algorithm, see Proposition \ref{prop:coefficients-bounded-by-det}), one can find a basis 
    \begin{align}\label{eq:basis-bounded-coeff}
        \bm f_1,\dots ,\bm f_{n-1} \in \Lambda (\idealq  ) 
    \end{align}
    whose coefficients as elements of $\Z ^{n-1}$ have absolute values $\le \totient (\idealq  )$. 
    
    Now we are ready to state the following consequence of Fourier analysis.

    \begin{proposition}[indicator function of a small convex set]\label{prop:Fourier-on-idele-class-group}
        Let $M,Y>1$ be real numbers.
    
        Let $P\subset C(\idealq  )/\R \baci_{>0}$ be a connected open subset such that its translate to $(C(\idealq  )/\R \baci_{>0})^0$ projects injectively by the map 
        \begin{align}
            (C(\idealq  )/\R \baci_{>0})^0 
            \to 
            \left( \frac{(\KR)\baci}{\OK\baci\cdot\Rpos } \right)^0
            \cong 
            (S^1)^{n-1}
        \end{align}
        and the image is
        a small convex subset of size $\ll _n \totient (\idealq  )^{-(n-2)}$ 
        in the sense of Definition \ref{def:small-convex-subsets}.

        Then one has a decomposition of the indicator function 
        $1_P\colon C(\idealq  )/\R \baci_{>0}\to \mathbf{C} 
        $
        of the form:
        \begin{align}
            1_P &= \sum _{\psi \in \Xi \subset (C(\idealq  )/\R \baci_{>0})\widehat{\ } }
            c_\psi \psi 
            +G+H,
        \end{align}
        where 
        \begin{itemize}
            \item 
            $\# \Xi = O_K (\totient (\idealq  ) Y^{n-1}) $,
            \item 
            $|c_\xi | \le 1$, 
            \item the values of $G$ have absolute value $\le 1$, and there is another small convex set $P'$ such that 
            $\oname{Supp} (G)\subset P'\setminus P$ and $\oname{Vol}_{\midele } (P'\setminus P)=O_n(1/M)$, 
            \item 
            $\lnorm H_\infty 
            \le O_n(M\frac{\log Y}{Y})$.
        \end{itemize}
        Moreover, for $\psi \in \left( \pi _0(C(\idealq  )/\R \baci_{>0})\right) \widehat{\ }\subset (C(\idealq  )/\R \baci_{>0})\widehat{\ } $
        we have 
        \begin{align}
            c_\psi = \frac{\vol _{\midele }(P)}{\vol _{\midele }(C(\idealq  )/\R \baci_{>0})} \ol {\psi (P)}+
            O_n(1/M).
        \end{align}
    \end{proposition}
    \begin{proof}
        We continue to identify $(C(\idealq  )/\R\baci_{>0})^0$ with $\R ^{n-1}/\Lambda (\idealq  )$ as in \eqref{eq:cover-is-iso-to}.
        Take a basis of $\Lambda (\idealq  )$ as in \eqref{eq:basis-bounded-coeff}.
        For this basis the matrix representing the linear map 
        $\phi\colon \Z ^{n-1}\cong \Lambda (\idealq  )\inj \Z ^{n-1}$ has entries of absolute values $\le \totient (\idealq  )$.
        It follows that the $\R $-linear inverse map $\phi_\R\inv \colon \R ^{n-1} \to \Lambda (\idealq  )\otimes _\Z \R \cong \R ^{n-1}$
        has coefficients $\ll _n \totient (\idealq  )^{n-2}$.
        It follows that if the constant $0<c_n<1$ is chosen small enough to match the implied constant, the inverse map $\phi_\R\inv $ sends the cube 
        \begin{align}
            [-c_n\totient (\idealq  )^{-(n-2)},c_n\totient (\idealq  )^{-(n-2)} ]^{n-1} 
        \end{align}
        into $(-1/5,1/5)^{n-1}$.
    
        Under such a choice of $\Lambda (\idealq  )\cong \Z ^{n-1}$ and $0<c_n<1$ we see that $P$ is a small convex set of size $<1/5$ in $C(\idealq  )/\R \baci_{>0}$ with respect to the corresponding trivialization $(C(\idealq  )/\R \baci_{>0})^0 \cong (S^1)^{n-1}$.
    
        Therefore Proposition \ref{prop:fact-from-Fourier} applies and gives the desired result.
        Note that by \eqref{eq:twofold-extension} 
        \begin{align}
            \# \pi _0(C(\idealq  )/\Rpos )\le h_K \cdot \totient (\idealq  ) \cdot \#\pi _0\left( \frac{(\KR)\baci}{\OK\baci\cdot\Rpos } \right) 
        \end{align}
        and $\#\pi _0\left( \frac{(\KR)\baci}{\OK\baci\cdot\Rpos } \right) \le 2^{r_1}=O_n(1)$.
    \end{proof}

    \subsubsection{A chart}
    To produce connected open sets $P$ as in Proposition \ref{prop:Fourier-on-idele-class-group}, it is useful to take the following kind of chart.
    
    To state it, for non-zero ideals $\ideala ,\idealq  \subset \OK $ let us denote by 
    $(\ideala /\idealq  \ideala )\baci $ the set of elements which is a generator of $\ideala /\idealq  \ideala $ as an $\OK /\idealq  $-module.
    For $\OK =\ideala $, this recovers the set of invertible elements $(\OK /\idealq  )\baci $.
    Note that by Chinese Remainder Theorem 
    $\ideala /\idealq  \ideala  = \dsum _{\idealp\divides\idealq  }\ideala /\idealp ^{n_\idealp }\ideala $,
    the condition that $\alpha \in \ideala /\idealq  \ideala $ is in $(\ideala /\idealq  \ideala )\baci $ is equivalent to the condition that 
    its image in $\ideala /\idealp ^{n_\idealp }\ideala $ is in $(\ideala /\idealp ^{n_\idealp }\ideala )\baci $ for all $\idealp \divides \idealq  $.
    
    Let $v_\idealp (\ideala ) \in \Z $ be the $\idealp $-adic valuation of a non-zero (fractional) ideal $\ideala $.
        Let us quickly recall its definition: the scalar extension $\ideala \calO _{K,\idealp }$ is a free $\calO _{K,\idealp }$-module of rank $1$. Take any generator $\alpha $ of $\ideala \calO _{K,\idealp }$ and we write $v_\idealp (\ideala ):= v_\idealp (\alpha )$, which is independent of the choice of $\alpha $.
    
        Then $(\ideala /\idealq  \ideala )\baci $ is equivalently defined as the set of residue classes $\ol\alpha \in \ideala /\idealq  \ideala $ such that some (or any) lift $\alpha \in \ideala $ satisfies 
        \begin{align}\label{eq:valuative-criterion}
            v_\idealp (\alpha ) = v_\idealp (\ideala )  
            \text{ for all }\idealp\divides\idealq  .
        \end{align}
        For each $\idealp $, this condition does not depend on the specific lift $\alpha $.

    Consider the map of sets
    \begin{align}\label{eq:chart-map}
        (\ideala /\idealq \ideala )\baci \to 
        \left( \dsum _{\idealp \divides \idealq  } K\baci / (1+\idealp ^{n_\idealp }\calO _{K,\idealp }) \right)
        \oplus 
        \dsum _{\idealp\notdivide\idealq  } \Z 
    \end{align}
    induced by $\ideala\setminus \{ 0\}   \inj K\baci $ 
    composed with $K\baci \surj K\baci / (1+\idealp ^{n_\idealp }\calO _{K,\idealp }) $ for each $\idealp\divides\idealq  $
    and 
    the valuation $v_\idealp \colon K\baci \surj \Z $ for each $\idealp\notdivide\idealq  $.
    The map \eqref{eq:chart-map} is well defined on $(\ideala /\idealq  \ideala )\baci $ because if 
    $\alpha _1,\alpha _2 \in \ideala $ are two elements satisfying 
    \eqref{eq:valuative-criterion}
    such that $\varepsilon := \alpha _1-\alpha _2\in \idealq  \ideala $,
    one can verify $\alpha _1/\alpha _2 \in 1+\idealp ^{n_\idealp }\calO_{K,\idealp }$ for $\idealp\divides\idealq  $.
    Indeed, we have 
    $(\alpha _1/\alpha _2)-1 = \varepsilon /\alpha _2 $
    and 
    \begin{align}
        v_\idealp (\varepsilon /\alpha _2)\ge (v_\idealp (\idealq  )+v_\idealp (\ideala ))-v_\idealp (\ideala ) = v_\idealp (\idealq  ) = n_\idealp .
    \end{align}

    \begin{proposition}\label{prop:chart}
        Let $\mathcal D  \subset (\KR )\baci$ be a fundamental domain for the $\OK \baci$-action and 
        $\{ \ideala _{\lambda } \} _{\lambda\in\ClK }$ be a complete set of representatives.

        The inclusions $\ideala_\lambda \setminus \{ 0\}  \inj K\baci $ induce the following bijections of sets 
        \begin{align}
            C(\idealq  ) \cong 
            \dunion _{\lambda\in\ClK}
            \frac{(\KR)\baci \times (\ideala_\lambda /\idealq  \ideala_\lambda )\baci }{\OK\baci }
            \cong 
            \dunion _{\lambda\in\ClK } \mathcal D  \times (\ideala_\lambda /\idealq  \ideala_\lambda )\baci.
        \end{align}
    \end{proposition}
    \begin{proof}
        Recall $C(\idealq  )=\widetilde C(\idealq  )/K\baci$.
        Consider the following maps 
        \begin{align}
            \widetilde C(\idealq  )\surj \dsum _{\idealp \text{ all}} \Z \surj \ClK 
        \end{align}
        and for each $\lambda\in \ClK $ let $\widetilde C(\idealq  )_\lambda $ be the fiber of this map over $\lambda $: 
        \begin{align}
            \widetilde C(\idealq  )_\lambda := \widetilde C(\idealq  ) \times _{\ClK } \{ \lambda \} .
        \end{align}
        Then we get a decomposition $\widetilde C(\idealq  ) = \dunion _{\lambda\in\ClK } \widetilde C(\idealq  )_\lambda $.
        Since the action of $K\baci$ preserves each $\widetilde C(\idealq  )_\lambda $, we get 
        \begin{align}\label{eq:chart-1}
            C(\idealq  ) = \dunion _{\lambda\in\ClK} \widetilde C(\idealq  )_\lambda /K\baci.
        \end{align}
    
        For a non-zero (fractional) ideal $\ideala $ we consider the following constructions. We define subsets
    \begin{align}
        K_{\idealq  ,\ideala }
        &:=
        \Bigl\{ ( 
                (x_\idealp )_{\idealp \divides \idealq  } ,
                (a_\idealp )_{\idealp \notdivide \idealq  } 
            )
        \in 
            \left( 
                \dsum _{\idealp \divides \idealq  }K\baci / (1+\idealp^{n_\idealp}\calO _{K,\idealp })
                \times \dsum _{\idealp \notdivide \idealq  } \Z 
            \middle) 
            \right| \\ 
            &\hspace{3cm}
            \forall \idealp\divides\idealq  :
               v_{\idealp }(x_\idealp ) = v_\idealp (\ideala ) ,
            \quad 
            \forall \idealp\notdivide\idealq  : 
            a_\idealp = v_\idealp (\ideala )
            \Bigr\} ,
            \\ 
            \widetilde C(\idealq  )_{\ideala }
            &:=
            (\KR )\baci \times K_{\idealq  ,\ideala  } \subset \widetilde C(\idealq  )_{[\ideala ] }, \text{ in particular }
            \widetilde C(\idealq  )_{\ideala _\lambda }
            \subset \widetilde C(\idealq  )_{\lambda }.
    \end{align}

        One sees that $\widetilde C(\idealq  )_{\ideala_\lambda }$ is stable under the $\OK \baci$-action. The induced map 
        \begin{align}\label{eq:chart-2}
            \widetilde C(\idealq  )_{\ideala_\lambda }/ \OK\baci 
            \to 
            \widetilde C(\idealq  )_\lambda / K\baci
        \end{align}
        is a bijection. The inverse is described as follows:
        let $(x;(x_\idealp )_{\idealp\divides\idealq  }; (a_\idealp )_{\idealp\notdivide\idealq  } )\in \widetilde C(\idealq  )_{\lambda }$.
        By the definitions of the sets $\widetilde C(\idealq  )_{\ideala_\lambda } \subset \widetilde C(\idealq  )_\lambda $ and 
        $\ClK = \oname{cok} 
        (K\baci \to \dsum _{\idealp \text{: all}} \Z ) $ 
        there is an $\alpha \in K\baci$ such that if we let $\bdiag (\alpha )\in \widetilde C(\idealq  )$ be its diagonal image 
        \begin{align}
            \bdiag (\alpha )\cdot ( x;( x_\idealp )_{\idealp\divides\idealq  }; (a_\idealp )_{\idealp\notdivide\idealq  } )
            \in 
            \widetilde C(\idealq  )_{\ideala _\lambda }.
        \end{align}
        This gives an element of $\widetilde C(\idealq  )_{\ideala_\lambda }$
        whose class in $\widetilde C(\idealq  )_{\ideala_\lambda }/ \OK\baci$
        is independent of the choice of $\alpha \in K\baci$.
    
        Since $\mathcal D  \subset (\KR)\baci$ is a fundamental domain for the $\OK\baci$-action we see that the composite 
        \begin{align}\label{eq:chart-3}
            \mathcal D  \times 
            K_{\idealq  ,\ideala_\lambda }
            \inj 
            \widetilde C(\idealq )_{\ideala_\lambda}
            \surj 
            \widetilde C(\idealq  )_{\ideala_\lambda}/\OK \baci
        \end{align}
        is a bijection.
    
        In view of the bijections \eqref{eq:chart-1} \eqref{eq:chart-2} \eqref{eq:chart-3},
        it remains to verify that the map \eqref{eq:chart-map}
        induces a bijection 
        \begin{align}
            (\ideala_\lambda /\idealq  \ideala_\lambda )\baci
            \isoto 
            K_{\idealq  ,\ideala _\lambda }
             .
        \end{align}
        Both sides naturally decompose into products $\prod _{\idealp\divides \idealq  }$.
        It suffices to show the following bijection %
        \begin{align}
            (\ideala_\lambda /\idealp ^{n_\idealp } \ideala_\lambda )\baci
            \isoto  
            \{ x\in K\baci /(1+\idealp ^{n_\idealp }\calO _{K,\idealp }) 
            \mid v_\idealp (x)=v_\idealp (\ideala_\lambda ) \} 
        \end{align}
        for each $\idealp\divides \idealq  $.
        
        The inverse map can be constructed as follows.
        Let $\alpha \in K\baci$ be an element satisfying $v_\idealp (\alpha )= v_\idealp (\ideala _\lambda )$.
        This implies $\alpha \in \ideala _\lambda \calO _{K,\idealp }$ and that it generates $\ideala _\lambda \calO _{K,\idealp }$ as an $\calO _{K,\idealp }$-module. So it determines a class in 
        $(\ideala _\lambda \calO _{K,\idealp }/ \idealp ^{n_\idealp }\ideala _\lambda \calO _{K,\idealp })\baci = (\ideala_\lambda /\idealp ^{n_\idealp } \ideala_\lambda )\baci$.
        Moreover, it is obvious that if the ratio of two such $\alpha ,\alpha '$ is in $1+\idealp ^{n_\idealp }\calO _{K,\idealp }$, 
        their difference is in $\idealp ^{n_\idealp }\ideala_\lambda \calO _{K,\idealp }$
        It is routine to check that this is actually the inverse. %
    \end{proof}

\section{Euclid's algorithm}\label{sec:Euclid}

The following statement is used in \S \ref{sec:trivialization}.

\newtheorem{propositionsubsec}[subsection]{Proposition}
\begin{propositionsubsec}\label{prop:coefficients-bounded-by-det}
    Let $\Lambda \subset \Z^{n}$ be a subgroup of finite index $\# (\Z ^n / \Lambda ) =D$.
    Then there is a basis $\bm v_1,\dots ,\bm v_{n}$ of $\Lambda $ whose entries have absolute values $\le D$.
\end{propositionsubsec}
\begin{proof}
    Take an arbitrary basis of $\Lambda $ and write it as a set of column vectors.

    The assertion is translated to the following assertion on matrices:
    let $A\in M_{n\times n}(\Z )$ be an $n\times n$ matrix with positive determinant $D$.
    Then there is an invertible matrix $U\in \GL _n (\Z )$ such that $AU$ has entries whose absolute values are $\le D$.

    When $n=1$ (or $n=0$), the assertion is trivial.
    We proceed by induction on $n$.

    Let $d_1>0$ be the greatest common divisor of the entries in the first row of $A$.
    By Euclid's algorithm, which can be realized as the right multiplication by invertible matrices, the matrix $A$ can be transformed into the form 
    \begin{align}\label{eq:after-first-transformation}
        \mtx{d_1 &0\dots 0 \\ *&A'} ,\quad \text{ with }A'\in M_{(n-1)\times (n-1)}(\Z ).
    \end{align}
    This also shows $d_1\divides D$.
    Set $\det (A')=D' = D/d_1$.

    From the induction hypothesis, by the right multiplication of invertible matrices which leaves the first column and row unchanged, one can transform $A'$ so that it has entries whose absolute values are $\le D'$.

    The entries $*$ in \eqref{eq:after-first-transformation} might {\it a priori} have absolute values $>D$, but one can use the entries of the modified $A'$ to make them smaller than $D'$.
This completes the proof.
\end{proof}


\bibliographystyle{plain}
\bibliography{WataruBibtex_arXiv_v3.bib}
\end{document}